\documentclass[onecolumn,12pt,draftcls]{IEEEtran}
\usepackage{cite}

\ifCLASSINFOpdf
\else
\fi
\usepackage[T1]{fontenc}
\usepackage{amsmath,amsfonts,amssymb,amscd}
\usepackage{graphicx,epsfig}
\usepackage{slashbox,pict2e}
\usepackage[font=singlespacing]{caption}
\usepackage{lscape}
\usepackage{verbatim}
\usepackage{epstopdf}
\usepackage{array}
\usepackage{subfigure}
\usepackage[ruled, vlined]{algorithm2e}
\usepackage{multirow}
\usepackage{bbm}
\usepackage{mathrsfs}
\usepackage{mathtools}
\usepackage{color}
\usepackage{float}
\usepackage{multirow}
\usepackage{amsthm}
\usepackage{bm}
\newtheorem{theorem}{Theorem}
\newtheorem*{remark}{Remark}
\newtheorem{lemma}{Lemma}

\DeclareMathOperator*{\argmin}{arg\,min}
\DeclareMathOperator*{\esssup}{ess\,sup}

\begin{document}
%
\title{Quickest Detection of Moving Anomalies in Sensor Networks}
%
%
%


\author{Georgios Rovatsos, ~\IEEEmembership{Student Member,~IEEE},\\ George V. Moustakides, ~\IEEEmembership{Fellow, ~IEEE}, \\ and Venugopal V. Veeravalli, ~\IEEEmembership{Fellow, ~IEEE}
\thanks{This work was supported in part by the Army Research Laboratory under Cooperative Agreement W911NF-17-2-0196 (IoBT CRA), and in part by the National Science Foundation (NSF) under grant CCF 16-18658 and CIF 15-14245, through the University of Illinois at Urbana-Champaign, and grant CIF 15-13373, through Rutgers University.

This work was presented in part in the 2019 Asilomar Conference on Signals, Systems, and Computers \cite{Rovatsos_asilomar:2019} and in the 2020 International Symposium on Information Theory \cite{Rovatsos_ISIT:2020}.}}
%
\maketitle

\maketitle

\begin{abstract}
The problem of sequentially detecting a moving anomaly which affects different parts of a sensor network with time is studied. Each network sensor is characterized by a non-anomalous and anomalous distribution, governing the generation of sensor data. Initially, the observations of each sensor are generated according to the corresponding non-anomalous distribution. After some \textit{unknown} but \textit{deterministic} time instant, a moving anomaly emerges, affecting different sets of sensors as time progresses. As a result, the observations of the affected sensors are generated according to the corresponding anomalous distribution. Our goal is to design a stopping procedure to detect the emergence of the anomaly as quickly as possible, subject to constraints on the frequency of \textit{false alarms}. The problem is studied in a quickest change detection framework where it is assumed that the evolution of the anomaly is \textit{unknown} but \textit{deterministic}. To this end, we propose a modification of Lorden's worst average detection delay metric to account for the trajectory of the anomaly that maximizes the detection delay of a candidate detection procedure. We establish that a \textit{Cumulative Sum}-type test solves the resulting sequential detection problem exactly when the sensors are \textit{homogeneous}. For the case of \textit{heterogeneous} sensors, the proposed detection scheme can be modified to provide a first-order asymptotically optimal algorithm. We conclude by presenting numerical simulations to validate our theoretical analysis.
\end{abstract}




\IEEEpeerreviewmaketitle

\section{Introduction}
\label{sec:intro}
In quickest change detection (QCD) \cite{tart-niki-bass-2014,poor-hadj-qcd-book-2009,veer-bane-elsevierbook-2013}, a sequentially observed time series undergoes a change in the underlying probability distribution at some unknown time instant. The goal is to design detection procedures, in the form of \textit{stopping times}, to detect this abrupt change as quickly as possible, subject to \textit{false alarm} (FA) constraints. It is of high importance that the proposed detection procedures are not solely implementable, but also offer strong theoretical guarantees with respect to defined delay-FA trade-off formulations.

In the classical, single-sensor QCD setting, the observations are initially \textit{independent and identically distributed} (i.i.d.) according to a known non-anomalous distribution. After some unknown time instant, which will be referred to as the \textit{changepoint}, a persistent change takes place and thereafter data is generated according to a known anomalous distribution. This model, referred to in the QCD literature as the \textit{i.i.d. model}, has been extensively studied under two frameworks that arise according to the underlying assumptions on the changepoint: i) the \textit{Bayesian} setting \cite{Shiryaev:1963,tart_veera:2005}, initially studied by Shiryaev, where the changepoint is modeled as a random variable of known probability distribution and the goal is to minimize an average detection delay, subject to constraints on the FA probability; ii) the \textit{minimax} setting \cite{lorden:1971,Pollak:1985,moustakides:1986,lai-ieeetit-1998}, where the changepoint is assumed to be \textit{unknown} but \textit{deterministic} and the goal is to minimize a worst-case (with respect to the changepoint) average detection delay, subject to a constraint on the \textit{mean time to false alarm} (MTFA).

In the case of multisensor networks, the theory of QCD has been widely employed to provide solutions to a variety of detection problems of interest. In such settings, the spatial evolution of the anomaly with time plays a crucial role, since different QCD problems with different solutions arise according the way sensors are affected by the anomaly. The simplest case corresponds to the anomaly persistently affecting a fixed set of sensors, the identity of which is known to the decision maker, after the changepoint. This problem is a trivial extension of the classical single-sensor QCD setting, hence, the algorithms in \cite{lorden:1971,Pollak:1985,moustakides:1986,lai-ieeetit-1998,Shiryaev:1963,tart_veera:2005} can be directly applied to provide performance guarantees. A significantly more complicated problem instance arises if we assume that the decision maker has no knowledge of the identity of the affected nodes. This problem has been extensively studied in the literature under the minimax setting \cite{tartakovsky2004change,1677904,mei2010efficient,Mei:2011,xie2013sequential,7890469,fellouris2016se}. Generalizations of these two aforementioned settings consider the case that the onset of the anomaly is perceived at different time instants across sensors \cite{Zou2017,hadjiliadis2009one,raghavan2010quickest,Ludkovski2012BayesianQD,Zou_sens:2018,Zhang:2018,Moustakides:2016,Rovatsos2:2016,7953065}. It is crucial to note that in the sensor network problems studied thus far, the core assumption that the anomaly persistently affects each sensor is made.

In this work, we study the problem of sequentially detecting a moving anomaly under Lorden's minimax framework \cite{lorden:1971}. In the moving anomaly QCD setting, it is assumed that different sets of nodes are affected by the anomaly as time progresses, and that the anomalous nodes are unknown to the decision maker. As a result, the anomaly does not affect any specific sensor persistently, but is persistent in the network as a whole. The problem was initially studied in \cite{Rovatsos_ISIT:2019,Rovatsos_SQA:2019}, where it was assumed that the anomaly evolves according to a \textit{discrete time Markov chain} and is of fixed size. In this paper, we lift the Markov assumption and assume that the trajectory of the anomaly is \textit{unknown} and \textit{deterministic}. To account for the lack of a specific model for the anomaly path, we modify Lorden's detection delay \cite{lorden:1971} to obtain a worst-path detection delay, and frame Lorden's QCD problem with the newly introduced delay metric. In the case of a network comprised of \textit{homogeneous} sensors, which share a common non-anomalous and a common anomalous distribution, we establish that a \textit{Cumulative Sum} (CUSUM)-type \cite{Page:1954} test that detects a transition to a mixture of distributions, each induced on the observations according to the identity of the anomalous nodes, is exactly optimal. Furthermore, we show that in the general case of \textit{heterogeneous} sensors the proposed test can be modified to provide a first-order asymptotically optimal solution.

The remainder of this paper is organized as follows. In Sec. \ref{sec:model}, we introduce necessary notation, describe the underlying statistical model for the observed process, and present the delay and FA metrics to be used along with the optimization problem to be solved. In Sec. \ref{sec:test_structure}, we introduce our proposed detection scheme. In Sec. \ref{sec:homogeneous}, we present the optimality of the proposed test for the case of homogeneous sensors. In Sec. \ref{sec:heterogeneous_network}, we show how to choose the parameters of the proposed detection procedure to derive a test that is first-order asymptotically optimal for a general, heterogeneous network. Finally, in Sec. \ref{sec:nums}, we conclude by providing simulation results to numerically validate the use of our proposed detection procedure.
\section{Problem Model}
\label{sec:model}
In this section, we present the statistical model that governs the data generated by the sensor network, as well as pose our QCD problem in a delay-FA optimization framework after introducing the worst-trajectory delay metric. We begin by introducing some necessary notation. Our convention in this work is that for any sequence $\{\alpha[k]\}_{k=1}^\infty$, and $k_2 > k_1$ we have that $\prod_{j=k_2}^{k_1} a[j] \triangleq 1$ and $\sum_{j=k_2}^{k_1} a[j] \triangleq 0$. Furthermore, for any sequence $\{\alpha[k]\}_{k=1}^\infty$, $\alpha[k_1,k_2] \triangleq \left[ \alpha[k_1],\ldots \alpha[k_2] \right]^\top$ denotes the samples from time $k_1$ to $k_2$. For a set $E$, $|E|$ denotes the number of elements in the set. Denote by $[L] \triangleq \{1,\dots,L\}$ a set of $L \geq 1$ sensors that comprise a sensor network monitored by a \textit{centralized} decision maker. Let $\{\bm{X}[k]\}_{k=1}^\infty$ denote the sequence of observations generated by the sensor network, where $\bm{X}[k] \triangleq [X_{1}[k],\dots,X_{L}[k]]^\top$ denotes the observation vector at time $k$ and $X_{\ell}[k] \in \mathbb{R}$ denotes the measurement obtained by sensor $\ell \in [L]$ at time $k$. Define by $\mathscr{F} \triangleq \{\mathscr{F}_k\}_{k=1}^\infty$ the filtration generated by the observation process, where $\mathscr{F}_k = \sigma(\bm{X}[1,k])$ denotes the $\sigma$-algebra generated by $\bm{X}[1,k]$. Furthermore, for $K \geq 0$ we use $\| \bm{x} \|_K$ to denote the $l_K$ norm of vector $\bm{x}$. Finally, for functions $f : \mathbb{R} \mapsto \mathbb{R}$, $g : \mathbb{R} \mapsto \mathbb{R}$ we have that $f(x) \sim g(x)$ denotes that $g(x) = f(x)(1+o(1))$ as  $x \rightarrow \infty$, where $o(1) \rightarrow \infty$ as $x \rightarrow \infty$.
\subsection{Observation Model}
Denote by $g_\ell(x)$, $f_\ell(x)$ the non-anomalous and anomalous \textit{probability density functions} (pdfs) at sensor $\ell \in [L]$, respectively. We assume that at each sensor the corresponding non-anomalous and anomalous distributions are different and that all data-generating distributions are known to the decision maker. Initially, all sensors generate data i.i.d. according to the non-anomalous mode, and observations are assumed to be independent across sensors. As a result, the joint pdf of $\bm{X}[k]$ is initially given by 
\begin{align}
\label{eq:pre_change_join}
g(\bm{X}[k]) \triangleq \prod\limits_{\ell=1}^L g_\ell(X_{\ell}[k]).
\end{align}

After some \textit{unknown} and \textit{deterministic changepoint} $\nu \geq 0$, a physical event leads to the emergence of a moving anomaly in the network. The anomaly moves around the network, affecting different sets of size $1 \leq m \leq L$ as time progresses. It is assumed that $m$ is constant and known to the decision maker. Define the process $\bm{S} \triangleq \{\bm{S}[k]\}_{k=1}^\infty$, where $\bm{S}[k]$ denotes the $m$-dimensional vector containing the indices of the anomalous nodes at time $k$. Note that for notational convenience, $\bm{S}[k]$ is defined for all $k\geq 1$ and not simply for $k > \nu$. We denote by $\mathcal{E}(L,m) \triangleq \mathcal{E} \triangleq \left\{\bm{E}_j\, \big|\, 1 \leq j \leq \binom Lm \right\}$ the set of all distinct possible vector-values that $\bm{S}[k]$ can take (WLOG we assume that the components of each vector are ordered to provide a unique vector per anomaly placement). Nodes affected by the anomaly generate observations according to the anomalous mode. In particular, for $k>\nu$, we have that conditioned on $\bm{S}$ the joint pdf of $\bm{X}[k]$ is given by
\begin{align}
p_{\bm{S}[k]}(\bm{X}[k])\triangleq   \left(\prod\limits_{\ell \in \bm{S}[k]} f_\ell(X_{\ell}[k])\right) \cdot\left(\prod\limits_{\ell \notin \bm{S}[k]} g_\ell(X_{\ell}[k])\right),
\end{align}
where for $\bm{E} \in \mathcal{E}$, $p_{\bm{E}}(\bm{x})$ denotes the joint pdf induced on a vector observation when the anomalous nodes are the ones contained in $\bm{E}$. We also assume that the observations are independent across time, conditioned on the changepoint. As a result, conditioned on $\nu$ and $\bm{S}$ the complete statistical model is the following: 
\begin{align}
\label{eq:distribution}
\bm{X}[k] \sim\left\{
\begin{array}{ll}
    g(\bm{X}[k]) &1\leq k \leq \nu \\
  p_{\bm{S}[k]}(\bm{X}[k])  & k > \nu.\\
\end{array}
\right. 
\end{align}
Furthermore, note that the aforementioned moving anomaly QCD problem can also be posed as the following dynamic composite hypothesis testing problem: at each time instant $k$, decide between the hypotheses
\begin{align}
\begin{split}
\label{eq:hypo}
&H^k_{1,\bm{S}}: \nu < k \text{ and anomaly evolves according to } \bm{S} \\
&H^k_0: \nu \geq k.
\end{split}
\end{align}  
The likelihood ratio corresponding to \eqref{eq:hypo} is then given by
\begin{align}
\label{eq:like_ratio}
\Gamma_{\bm{S}}(k,\nu) \triangleq \prod_{j=\nu+1}^k \left( \prod\limits_{\ell \, \in \, \bm{S}[j]} \frac{f_\ell(X_\ell[j])}{g_\ell(X_\ell[j])} \right) = 
\prod_{j=\nu+1}^k \Gamma_{\bm{S}}(j,j-1).
\end{align}
\subsection{Delay-FA Trade-off Formulation}
In this work, the goal is to design a detection procedure in the form of a \textit{stopping time} to detect the abrupt change in distribution detailed in \eqref{eq:distribution}. A stopping time $\tau$ \cite{tart-niki-bass-2014,poor-hadj-qcd-book-2009,veer-bane-elsevierbook-2013} adapted to $\mathscr{F}$ is a positive random variable which satisfies $\{ \tau \leq k \} \in \mathscr{F}_k$ for all $k \geq 1$, i.e., the decision to raise an alarm at time $k$ is determined only by the observations up to that point. An efficient stopping procedure offers quick detection by guaranteeing a sufficiently low frequency of false alarms. To frame this trade-off mathematically, we employ a modified version of Lorden's delay-FA formulation \cite{lorden:1971}. In particular, since the anomaly trajectory process $\bm{S}$ is assumed to be unknown, we modify Lorden's delay metric to evaluate candidate detection schemes according to the anomaly path that maximizes the expected detection delay. In particular, denote by $\mathbb{E}_\nu^{\bm{S}}[\cdot]$ the expectation when the changepoint is equal to $\nu$ and the trajectory of the anomaly is specified by $\bm{S}$. Then, for any stopping rule $\tau$ adapted to $\mathscr{F}$ consider the following modification of Lorden's worst average detection delay ($\mathrm{WADD}$) metric:
\begin{align}
\label{eq:delay_metric}
\mathrm{WADD}(\tau) = \sup_{\bm{S}}  \sup_{\nu \geq 0} \esssup \mathbb{E}_\nu^{\bm{S}} \left[\tau -\nu  |\tau > \nu, \mathscr{F}_{\nu } \right],
\end{align}
where the convention that $\mathbb{E}_\nu^{\bm{S}} \left[\tau -\nu |\tau > \nu , \mathscr{F}_{\nu } \right] \triangleq 1$ when $\mathbb{P}_\nu^{\bm{S}}(\tau > \nu)=0$ is used. Note that an additional sup is used to account for the trajectory of the anomaly that maximizes the detection delay of $\tau$. Denote by $\mathbb{E}_\infty[\cdot]$ the expectation when no anomaly is present. To quantify the frequency of FA events we use the \textit{mean time to false alarm} (MTFA), denoted by $\mathbb{E}_\infty[\tau]$ for stopping time $\tau$. For $\gamma >1$ a pre-determined constant, define the class of stopping times
\begin{align}
\label{eq:c_gamma}
\mathcal{C_\gamma} \triangleq \{\tau : \mathbb{E}_\infty[\tau] \geq \gamma\}.
\end{align} 
Our goal then is to design a stopping time $\tau$ to solve the following stochastic optimization problem:
\begin{equation}
\label{eq:optimization}\
\begin{aligned}
& \underset{\tau}{\text{min}}
& & \mathrm{WADD}(\tau) \\
& \text{s.t.}
& & \tau \in \mathcal{C_\gamma}.
\end{aligned}
\end{equation}
\subsection{Randomized Anomaly Allocation Model}
Before proceeding to the presentation of our main theoretical results, it is necessary to introduce another statistical model that plays an important role in the mathematical analysis, as well as in the intuitive interpretation of the results. In particular, consider an alternate setting to that of \eqref{eq:distribution}, where at each time instant after the changepoint the $m$ anomalous nodes are chosen randomly. To this end, denote by $\bm{\alpha} = \left\{\alpha_{\bm{E}} : \bm{E} \in \mathcal{E}\right\} \in \mathcal{A}$ the \textit{probability mass function} (pmf) containing the probabilities that each of the vectors in $\mathcal{E}$ is chosen as the vector of anomalous nodes. I.e., at each time instant $k$ the probability that the $m$ anomalous nodes are chosen to be in $\bm{E}$ is given by $\alpha_{\bm{E}}$, and the set of anomalous nodes are picked i.i.d. across time. Here, $\mathcal{A}$ denotes the simplex of all probability vectors of dimension $|\mathcal{E}|$. When at each time instant after the changepoint the anomalous nodes are placed i.i.d. randomly according to $\bm{\alpha}$, we have that the induced pdf after the changepoint is a mixture of pdfs given by
\begin{align}
\overline{p}_{\bm{\alpha}}(\bm{X}[k])\triangleq \sum\limits_{\bm{E}\, \in \, \mathcal{E}} \alpha_{\bm{E}} p_{\bm{E}}(\bm{X}[k]).
\end{align}
As a result, the complete observation model for the case of a randomized anomaly allocation according to pmf $\bm{\alpha}$ is the following
\begin{align}
\label{eq:eq:random_stat_model}
\bm{X}[k] \sim\left\{
\begin{array}{ll}
    g(\bm{X}[k])  & 1\leq k \leq \nu \\
  \overline{p}_{\bm{\alpha}}(\bm{X}[k])& k > \nu.\\
\end{array}
\right.
\end{align}
Similarly to \eqref{eq:distribution}, we can pose the following dynamic composite hypothesis testing problem corresponding to \eqref{eq:eq:random_stat_model}: at each time $k$ choose between the hypotheses
\begin{align}
\begin{split}
\label{eq:hypo_2}
&\bar{H}^k_{1,\bm{\alpha}}: \nu < k \text{ and anomaly placed randomly according to } \bm{\alpha} \\
&\bar{H}^k_0: \nu \geq k.
\end{split}
\end{align}
The likelihood ratio corresponding to \eqref{eq:hypo_2} is then given by
\begin{align}
\label{eq:ave_like_ratio}
\mathcal{L}_{\bm{\alpha}}(k,\nu) &\triangleq  
\prod_{j=\nu+1}^k \frac{\overline{p}_{\bm{\alpha}}(\bm{X}[j])}{g(\bm{X}[j])}  
= \prod_{j=\nu+1}^k \left( \sum_{\bm{E} \,\in\, \mathcal{E}} \alpha_{\bm{E}} \prod\limits_{\ell \,\in \, \bm{E}}\frac{ f_\ell(X_{\ell}[j])}{g_\ell(X_\ell[j])} \right) 
= \prod_{j=\nu}^k  \mathcal{L}(j,j-1).
\end{align}
Furthermore, consider the Kullback-Leibler (KL) number between the non-anomalous and anomalous distributions in \eqref{eq:eq:random_stat_model} given by
\begin{align}
\label{eq:KL}
I_{\bm{\alpha}}\triangleq \overline{\mathbb{E}}^{\bm{\alpha}}_{0} \left[\log \frac{\overline{p}_{\bm{\alpha}}(\bm{X}[1])}{g(\bm{X}[1])} \right],
\end{align}
where $\overline{\mathbb{E}}^{\bm{\alpha}}_{\nu}[\cdot]$ denotes the expectation when the underlying statistical model is that of \eqref{eq:eq:random_stat_model} with changepoint being equal to $\nu$ and the anomaly placed randomly according to $\bm{\alpha}$. 

Note that the model in eq. \eqref{eq:eq:random_stat_model} characterizes a different QCD problem compared to the one described in eqs. \eqref{eq:distribution} - \eqref{eq:optimization}, one in which the non-anomalous and anomalous pdfs are completely specified. This QCD problem is associated with a corresponding detection delay. In particular, for stopping time $\tau$, define the detection delay corresponding to the QCD problem detailed in \eqref{eq:eq:random_stat_model} by
\begin{align}
\label{eq:random_delay}
\overline{\mathrm{WADD}}_{\bm{\alpha}}(\tau) =  \sup_{\nu \geq 0} \esssup \overline{\mathbb{E}}^{\bm{\alpha}}_{\nu}[\tau - \nu |  \tau > \nu,\mathscr{F}_{\nu }].
\end{align}
Here, we also use the convention that $\overline{\mathbb{E}}^{\bm{\alpha}}_{\nu}[\tau - \nu |  \tau > \nu,\mathscr{F}_{\nu }] \triangleq 1$ when $\overline{\mathbb{P}}^{\bm{\alpha}}_\nu(\tau > \nu) =0$. Since both the non-anomalous and anomalous joint pdfs for the QCD problem presented in \eqref{eq:eq:random_stat_model} - \eqref{eq:random_delay} are completely specified, the classical CUSUM test studied in \cite{lorden:1971,Pollak:1985,moustakides:1986,lai-ieeetit-1998} can be directly applied to solve this QCD problem exactly \cite{moustakides:1986}. In the remainder of the paper, we show that solving the QCD problem in \eqref{eq:eq:random_stat_model} - \eqref{eq:random_delay} for a specific choice of $\bm{\alpha}$, which depends in the data generating distributions of the sensors, can lead to the solution of our QCD problem of interest.
\section{Universal Asymptotic Lower Bound on WADD}
s
\section{Proposed Detection Algorithm}
\label{sec:test_structure}
As discussed in Sec. \ref{sec:model}, in this work we establish that the solution to \eqref{eq:distribution} - \eqref{eq:optimization} is the solution to the QCD problem outlined in \eqref{eq:eq:random_stat_model} - \eqref{eq:random_delay} for a specific choice of $\bm{\alpha}$. To this end, we will focus on the analysis of the CUSUM test corresponding to the QCD problem in \eqref{eq:eq:random_stat_model} - \eqref{eq:random_delay}. In particular, for $\bm{\lambda} \in \mathcal{A}$, consider the following Mixture-CUSUM (M-CUSUM) test statistic 
\begin{align}
\label{eq:test_stat}
W_{\bm{\lambda}}[k] \triangleq \max\limits_{1\leq i \leq k} \mathcal{L}_{\bm{\lambda}}(k,i-1),
\end{align}
with the corresponding stopping time 
\begin{align}
\label{eq:stop_time}
\tau_W (\bm{\lambda}, b) \triangleq  \inf\left\{ k \geq 1 : W_{\bm{\lambda}}[k] \geq e^b \right\},
\end{align}
where $b>0$ a constant chosen so that the stopping time satisfies the FA constraint in \eqref{eq:c_gamma}. It can be easily established, (see, e.g, \cite{veer-bane-elsevierbook-2013}) that the test statistic in \eqref{eq:test_stat} can be computed recursively through the recursion
\begin{align}
\label{eq:test_rec}
W_{\bm{\lambda}}[k] =  \max\{W_{\bm{\lambda}}[k-1] ,1 \}\mathcal{L}_{\bm{\lambda}}(k,k-1) ,
\end{align}
where $W_{\bm \lambda}[0] \triangleq 0$ for any $\bm{\lambda} \in \mathcal{A}$. Note that the M-CUSUM test presented in eqs. \eqref{eq:test_stat} - \eqref{eq:test_rec} is the exact solution to the QCD problem detailed in \eqref{eq:eq:random_stat_model} - \eqref{eq:random_delay} when $\bm{\alpha} = \bm{\lambda}$, if $b$ is chosen such that $\mathbb{E}[\tau_W (\bm{\lambda}, b)] = \gamma$ \cite{moustakides:1986}. In the remainder of the paper, we establish that by choosing $\bm{\lambda}$ accordingly the M-CUSUM procedure is also an exact solution to \eqref{eq:optimization} when the network is comprised of homogeneous sensors, as well as first-order asymptotically optimal for the general heterogeneous network case. Our analysis is based on relating the two QCD models presented in Sec. \ref{sec:model} and exploiting tools used for the analysis of the CUSUM test in \cite{moustakides:1986}, 
\cite{lai-ieeetit-1998}. We begin by presenting an important theorem relating the detection delay metrics \eqref{eq:delay_metric}, \eqref{eq:random_delay} introduced in Sec. \ref{sec:model}.
\begin{theorem}
\label{ineqs_lemma}
Let $\gamma > 1$ and $\bm{\alpha} \in \mathcal{A}$. For the M-CUSUM test introduced in eqs. \eqref{eq:test_stat} - \eqref{eq:test_rec} with $b$ chosen such that $\mathbb{E}_\infty[\tau_W(\bm{\alpha}, b)] = \gamma$ we have that 
\begin{align}
\label{eq:ineqs}
\mathrm{WADD}(\tau_W(\bm{\alpha}, b)) \geq \inf\limits_{\tau \in C_\gamma}\mathrm{WADD}(\tau) \geq   \overline{\mathrm{WADD}}_{\bm{\alpha}}(\tau_W(\bm{\alpha}, b)) .
\end{align}
\end{theorem}
\begin{proof}
The proof of the theorem is based on Lemmas \ref{thm:lemma_trunc} and \ref{converge} which are introduced and proved in the Appendix. The complete proof follows the analysis in \cite{moustakides:1986} and is provided in the Appendix.
\end{proof}
\begin{remark}
In view of Lemma \ref{thm:lemma_trunc}, we will WLOG be considering stopping times $\tau$ satisfying $\mathbb{E}_\infty[\tau] <\infty$, since any stopping time that does not satisfy this condition can be truncated to provide a smaller detection delay while at the same time satisfying the FA constraint.
\end{remark}
\section{Homogeneous Sensor Network Case}
\label{sec:homogeneous}
In this section, we consider the case of a homogeneous sensor network, i.e., a network where $g_\ell(x) \triangleq g(x)$ and $f_\ell(x) \triangleq f(x)$ for all $\ell \in [L]$, $x \in \mathbb{R}$ (note that $g(x)$ denotes the common marginal non-anomalous pdf while $g(\bm{x})$ is used to denote the joint pdf under $\mathbb{P}_\infty(\cdot)$). Since the network is symmetric, an intuitive weight choice is to choose all the weights in the M-CUSUM of Sec. \ref{sec:test_structure} test to be equal. This then implies that by the symmetry of the statistical model, as well as the resulting symmetry of the detection procedure with respect to the placement of the anomaly, placing the anomaly randomly or with the worst-path approach will not lead to a different detection delay. In particular, we have the following lemma:
\begin{lemma}
\label{symmetry_lemma}
Consider a homogeneous sensor network where $g_\ell(x) \triangleq g(x)$ and $f_\ell(x) \triangleq f(x)$ for all $\ell \in [L]$, $x \in \mathbb{R}$. Let $\bm{\lambda}_U \triangleq \left[\binom Lm, \ldots, \binom Lm \right]^\top$ the uniform M-CUSUM weights vector. For any threshold $b>0$ and any $\bm{\alpha} \in \mathcal{A}$ we have that
\begin{align}
\mathrm{WADD}(\tau_W(\bm{\lambda}_U, b)) = \overline{\mathrm{WADD}}_{\bm{\alpha}}(\tau_W(\bm{\lambda}_U, b)).
\end{align}
\end{lemma}
\begin{proof}
See Appendix.
\end{proof}
By using Theorem \ref{ineqs_lemma} and Lemma \ref{symmetry_lemma} we can establish the exact optimality of the M-CUSUM test with uniform weights for the case of a homogeneous sensor network.
\begin{theorem}
\label{optimality}
Consider a homogeneous sensor network where $g_\ell(x) \triangleq g(x)$ and $f_\ell(x) \triangleq f(x)$ for all $\ell \in [L]$, $x \in \mathbb{R}$. Let $\gamma > 1$. The M-CUSUM test with uniform weights $\bm{\lambda} = \bm{\lambda}_U\triangleq \left[\binom Lm, \ldots, \binom Lm \right]^\top$ and threshold $b$ chosen such that $\mathbb{E}_\infty[\tau_W(\bm{\lambda}_U, b)] = \gamma$ is exactly optimal with respect to \eqref{eq:optimization}, i.e.,
\begin{align}
\mathrm{WADD}(\tau_W(\bm{\lambda}_U, b)) = \inf\limits_{\tau \in C_\gamma}\mathrm{WADD}(\tau).
\end{align}
\end{theorem}
\begin{proof}
The result follows directly by combining Theorem \ref{ineqs_lemma} and Lemma \ref{symmetry_lemma}.
\end{proof}
Theorem \ref{optimality} implies that, for the case of homogeneous sensors, the M-CUSUM test that solves the QCD problem of eqs. \eqref{eq:eq:random_stat_model} - \eqref{eq:random_delay} for a uniform pmf $\bm{\alpha}=\bm{\lambda}_U$ is also the exact solution to \eqref{eq:distribution} - \eqref{eq:optimization}. Next, we investigate whether a similar result holds for the general case of heterogeneous networks.
\section{Heterogeneous Sensor Network Case}
\label{sec:heterogeneous_network}
In Sec. \ref{sec:homogeneous}, we saw how the symmetry of a homogeneous sensor network can facilitate the construction of an exactly optimal test with respect to \eqref{eq:optimization}. However, in the case of a heterogeneous sensor network, such a symmetry is no longer valid, and a similar lemma to Lemma \ref{symmetry_lemma} can not be established in general. As a result, symmetry cannot be employed to deive an upper bound on the detection delay of our proposed M-CUSUM test. In this section, we show that by choosing the weight of the M-CUSUM test accordingly, a first-order asymptotically optimal test can be derived by exploiting an asymptotic type of symmetry that is related to the expected drift of our test statistic. 
\subsection{Universal Asymptotic Lower Bound on the WADD}
We begin our analysis by presenting an asymptotic lower bound on $\mathrm{WADD}$ for stopping times satisfying the false alarm constraint $\mathbb{E}_\infty[\tau] \geq \gamma$. Our lower bound is derived by using Theorem \ref{ineqs_lemma} together with the asymptotic lower bound on $\overline{\mathrm{WADD}}$ \cite{lorden:1971,lai-ieeetit-1998}. In particular, note that the inequalities in Theorem \ref{ineqs_lemma} hold for any arbitrary $\bm{\alpha} \in \mathcal{A}$. Hence, to get the tightest asymptotic lower bound we need to consider the $\bm{\alpha}$ that maximizes the coefficient of the asymptotic rate of $\overline{\mathrm{WADD}}$. To this end, define the minimizer of the effective KL number $I_{\bm{\alpha}}$ by 
\begin{align}
\label{eq:KL_MINI}
\bm{\alpha}^* \triangleq \argmin\limits_{\bm{\alpha}\,\in\,\mathcal{A}} I_{\bm{\alpha}}.
\end{align}
It can be shown that $I_{\bm{\alpha}}$ is strictly convex with respect to $\bm{\alpha}$, hence, such a minimizer is uniquely defined. We then have the following theorem:
\begin{theorem}
\label{theorem:lower_bound}
Let $\bm{\alpha}^*$ defined as in \eqref{eq:KL_MINI}. We then have that
\begin{align}
\inf\limits_{\tau \,\in \,\mathcal{C}_\gamma}\mathrm{WADD}(\tau) \geq \frac{\log \gamma}{I_{\bm{\alpha}^*}}(1+o(1))
\end{align}
  as $\gamma \rightarrow \infty$.
\end{theorem}
\begin{proof}
By Theorem \ref{ineqs_lemma} we have that for any $\bm{\alpha} \in \mathcal{A}$ and any $\gamma>1$
\begin{align}
\label{eq:a_simple_eq}
\inf\limits_{\tau \,\in \,\mathcal{C}_\gamma} \mathrm{WADD}(\tau) \geq \inf\limits_{\tau \,\in \,\mathcal{C}_\gamma} \overline{\mathrm{WADD}}_{\bm{\alpha}}(\tau).
\end{align}
which implies that the inequality also holds for $\bm{\alpha} = \bm{\alpha}^*$ , i.e., 
\begin{align}
\inf\limits_{\tau \,\in\, \mathcal{C}_\gamma} \mathrm{WADD}(\tau) \geq \inf\limits_{\tau \,\in \,\mathcal{C}_\gamma} \overline{\mathrm{WADD}}_{\bm{\alpha}^*}(\tau) \sim \frac{\log \gamma}{I_{\bm{\alpha}^*}},
\end{align}
where the asymptotic delay approximation follows from the asymptotic analysis of the CUSUM test \cite{lorden:1971,lai-ieeetit-1998}.
\end{proof}
\subsection{Asymptotic Upper Bound on the WADD of M-CUSUM Test}
\label{sec:upper_bound}
Although deriving a lower bound on $\mathrm{WADD}$ is similar for both homogeneous and heterogeneous sensor networks (Theorem \ref{ineqs_lemma}), upper bounding $\mathrm{WADD}$ in the latter case for any $\bm \lambda$ is nontrivial. To find the weights of the M-CUSUM test that result in an asymptotically optimal test, it is important to further investigate the minimization of $I_{\bm{\alpha}}$. To this end, we present the following lemma:
\begin{lemma}\label{equal_drifts_lemma}
Let $\bm{\alpha}^*$ defined as in \eqref{eq:KL_MINI}. We then have that:

i) Case $m \geq 2$: $\bm{\alpha}^*$ cannot be a corner point of $\mathcal{A}$, i.e., $2 \leq \|\bm{\alpha}^*\|_0 \leq |\mathcal{E}|$. 

If $\|\bm{\alpha}^*\|_0 = |\mathcal{E}|$ (interior-point minimum), we have that
\begin{align}
\label{eq:equal_drifts}
\mathbb{E}_{p_{\bm E}}\left[\log\left(\frac{\overline{p}_{\bm{\alpha}^*}(\bm{X})}{g(\bm{X})} \right)\right] = \mathbb{E}_{p_{\bm{E}'}}\left[\log\left(\frac{\overline{p}_{\bm{\alpha}^*}(\bm{X})}{g(\bm{X})} \right)\right]
\end{align}
for all $\bm{E}$, $\bm{E}' \in \mathcal{E}$, where $\mathbb{E}_{p_{\bm E}}[\cdot]$ denotes the expected value when anomalous nodes are given in $\bm{E} \in \mathcal{E}$. 

If $2 \leq \|\bm{\alpha}^*\|_0 < |\mathcal{E}|$ (boundary-point minimum), let $\mathcal{E}' \triangleq \{\bm{E} \in \mathcal{E} : \alpha^*_{\bm{E}} > 0 \}$ the subset of vectors in $\mathcal{E}$ for which non-zero weights are assigned in $\bm{\alpha}^*$. We then have that for all $\bm{E}$, $\bm{E}' \in \mathcal{E}'$ eq. \eqref{eq:equal_drifts} holds. Furthermore, we have that for all $\bm{B} \in \mathcal{E}'$, $\bm{B}' \in \mathcal{E} \,\char`\\ \, \mathcal{E}'$
\begin{align}
\label{eq:unequal_drifts}
\mathbb{E}_{p_{\bm B'}}\left[\log\left(\frac{\overline{p}_{\bm{\alpha}^*}(\bm{X})}{g(\bm{X})} \right)\right] > \mathbb{E}_{p_{\bm{B}}}\left[\log\left(\frac{\overline{p}_{\bm{\alpha}^*}(\bm{X})}{g(\bm{X})} \right)\right].
\end{align}
ii) Case $m = 1$ (single anomalous node): $\bm{\alpha}^*$ is an interior point of $\mathcal{A}$, i.e., $\|\bm{\alpha}^*\|_0 = |\mathcal{E}| = L$.
\end{lemma}
\begin{proof}
See Appendix.
\end{proof}

By exploiting the properties presented in Lemma \ref{equal_drifts_lemma}, we derive an asymptotic upper bound on $\mathrm{WADD}(\tau_W(\bm{\alpha}^*, b))$. In particular, we have the following theorem:
\begin{theorem}
\label{theorem:upper_bound}
Let $\bm{\alpha}^*$ defined as in \eqref{eq:KL_MINI}. Assume that 
\begin{align}
\label{eq:assumptio_mom}
\max_{\bm{E} \,\in\, \mathcal{E}}\mathbb{E}_{p_E}\left[\left(\log\frac{\overline{p}_{\bm{\alpha}^*}(\bm{X})}{g(\bm{X})}  \right)^2   \right]<\infty
\end{align}
We then have that as $b \rightarrow \infty$
\begin{align}
\mathrm{WADD}(\tau_W(\bm{\alpha}^*, b)) \leq  \frac{b}{I_{\bm{\alpha}^*}}(1 + o(1)).
\end{align}
\end{theorem}
\begin{proof}
Proof is based on Lemma \ref{equal_drifts_lemma} and the analysis in \cite{lai-ieeetit-1998} and is provided in the Appendix.
\end{proof}
\subsection{Asymptotic Optimality of M-CUSUM Test}
By combining Theorems \ref{theorem:lower_bound} with \ref{theorem:upper_bound} we can establish the asymptotic optimality of the M-CUSUM test for weight choice $\bm{\lambda} = \bm{\alpha}^*$.
\begin{theorem}
\label{theorem:as_opt}
Let $\bm{\alpha}^*$ defined as in \eqref{eq:KL_MINI}. Assume that 
\begin{align}
\max_{\bm{E} \,\in\, \mathcal{E}}\mathbb{E}_{p_{\bm E}}\left[\left(\log\frac{\overline{p}_{\bm{\alpha}^*}(\bm{X})}{g(\bm{X})}  \right)^2   \right]<\infty.
\end{align}
We then have that:

i) For any $\gamma > 1$, $\mathbb{E}_\infty[\tau_W(\bm{\alpha}^*, \log \gamma)]  \geq \gamma$.

ii)
\begin{align}
\inf\limits_{\tau \in \mathcal{C}_\gamma}\mathrm{WADD}(\tau) \sim \mathrm{WADD}(\tau_W(\bm{\alpha}^*, \log \gamma)) \sim  \frac{\log \gamma}{I_{\bm{\alpha}^*}}
\end{align}
as $\gamma \rightarrow \infty$.
\end{theorem}
\begin{proof}
i) Follows directly from the MTFA analysis of the CUSUM test \cite{lorden:1971, lai-ieeetit-1998}.

ii) Follows from i) and Theorems \ref{theorem:lower_bound} and \ref{theorem:upper_bound}.
\end{proof}
Essentially, Theorem \ref{theorem:as_opt} implies that, for the case of heterogeneous sensors, there exists a choice of $\bm{\alpha}$ such that the M-CUSUM test that solves the QCD problem of eqs. \eqref{eq:eq:random_stat_model} - \eqref{eq:random_delay} for said $\bm{\alpha}$ exactly is also asymptotically optimal with respect to \eqref{eq:distribution} - \eqref{eq:optimization}. This $\bm{\alpha}$ is the one that minimizes the KL-number in \eqref{eq:KL} and depends on the data-generating distributions of the sensors.

The asymptotic optimality of the M-CUSUM test with weights $\bm{\alpha}^*$ can be intuitively explained through Lemma \ref{equal_drifts_lemma}. In particular, since large $\gamma$ implies a large threshold, if we consider the logarithm of the M-CUSUM test statistic in \eqref{eq:test_stat}, the expected drift of the added log-likelihood ratio dominates the asymptotic performance of the M-CUSUM test. For a general choice of $\bm{\lambda}$ this expected drift is not generally equal throughout the different anomaly placements $\bm{E} \in \mathcal{E}$. In such a case, the worst-path delay will be dominated by the smallest resulting drift among anomaly placements. However, by Lemma \ref{equal_drifts_lemma} we know that choosing $\bm{\lambda} = \bm{\alpha}^*$ implies that the drift of the statistic is equal among a specific subset of anomaly placements, and also equals to the inverse of the best possible asymptotic rate coefficient. Furthermore, we saw in Lemma \ref{equal_drifts_lemma} that all other placements lead to a larger drift, hence, do not play a role asymptotically due to the worst-path aspect of the delay. As a result, we have that the delay rate of our proposed test will match the universally best rate.
\section{Numerical Results}
In this section, we conduct numerical simulations for the studied moving anomaly QCD problem for the case of a single anomalous node ($m = 1$) and different network sizes $L$. We present results for both homogeneous and heterogeneous sensor networks.

For the case of a homogeneous network, we assume that $g = \mathcal{N}(0,1)$ and $f = \mathcal{N}(1,1)$. For homogeneous networks, we can introduce two additional tests that can be used as a comparison: a heuristic test; and an oracle-type test.
In particular, note that for all $\bm{S}$ we have that
\begin{align*}
&\mathbb{E}_\infty\left[\sum\limits_{\ell=1}^L \log\frac{f(X_{\ell}[k])}{g(X_{\ell}[k])} +(L-m)D(f\|g)\right] = -mD(f\|g)<0
\\& \mathbb{E}_0^{\bm{S}}\left[\sum\limits_{\ell=1}^L \log\frac{f(X_{\ell}[k])}{g(X_{\ell}[k])} +(L-m)D(f\|g)\right] = m D(f\|g)>0.
\end{align*}
This suggests that the following Naive-CuSum (N-CuSum) test may be a candidate test for detecting the distribution change described in \eqref{eq:distribution}. In particular, consider the test described by the following recursion:

{{
\small
\begin{align}
W_N[k]  \triangleq \left( W_{N}[k-1] + \sum\limits_{\ell=1}^L \log\frac{f(X_{\ell}[k])}{g(X_{\ell}[k])} +(L-m)D(f\|g)\right)^+
\end{align}
}}
with $W_N[0]  \triangleq 0$ and corresponding stopping time 
\begin{flalign*}
\mathcal{\tau}_N =  \inf\left\{ k \geq 1 : W_N[k] \geq b \right\}.
\end{flalign*}
Although the N-CuSum test can be employed to detect the anomaly because of having a positive expected drift, it does not necessarily solve the QCD problem in \eqref{eq:optimization}. 

We also compare our proposed procedure to an Oracle-CUSUM (O-CUSUM) test, which is a CUSUM test that uses complete knowledge of $\bm{S}$. I.e, to define this test we assume that at time $k$ we do not know whether a change has occured, but we know which set of sensors would be affected if an anomaly had already emerged in the network. In particular, consider the statistic calculated by using the following recursion:

\begin{align}
W_O[k] =  \left( W_{O}[k-1] + \log \left(\prod_{\ell \, \in \,\bm{S}[k] }\frac{f(X_{\ell}[k])}{g(X_{\ell}[k])} \right)\right)^+  
\end{align}
with $W_O[0] \triangleq 0$ and with corresponding stopping time 
\begin{align}
\tau_O =  \inf\left\{ k \geq 1 : W_O[k] \geq b \right\}.
\end{align}
Since this O-CuSum test uses the knowledge of the location of the anomalous nodes, it is expected to perform better than our proposed test. However, such a test is not tractable since in practice such location information will not be available to the decision maker.

In Figs. \ref{fig:L-5}, \ref{fig:L-10} and \ref{fig:L-20} we compare the M-CUSUM test, with the N-CUSUM test and the O-CUSUM test for network sizes $L=5$, $L=10$ and $L=20$. Note that due to the symmetry of the M-CUSUM and the N-CUSUM test, $\mathrm{WADD}$ is equal to the delay for any arbitrary path of the anomaly. By inspecting Figs. \ref{fig:L-5}, \ref{fig:L-10} and \ref{fig:L-20} we note that the M-CUSUM test outperforms the heuristic N-CUSUM test, which is expected since the M-CUSUM test is optimal with respect to \eqref{eq:optimization}. In addition, we note that the O-CUSUM test performs better than the other detection schemes, which is to be expected since it exploits complete knowledge of $\bm{S}$. We also note that as $L$ increases the performance gap between the O-CUSUM test and the M-CUSUM test increases. This is because as the network size increases the noise that is introduced in the M-CUSUM test due to nodes that are not anomalous also increases. This is not the case for the O-CUSUM test, since this scheme inherently assumes complete knowledge of the anomalous nodes. In Fig. \ref{fig:different_L_comparison}, we evaluate the performance of our proposed M-CUSUM test for different values of $L$. We note that as $L$ increases our proposed test performs worse, which is expected since the algorithm is affected by more noise from non-anomalous nodes for larger network sizes.
\label{sec:nums}
\begin{figure*}[t]
\centering
\subfigure[$\mathrm{WADD}$ versus MTFA for $L=5$, $m=1$.]{
\label{fig:L-5}
\mbox{
\epsfig{file=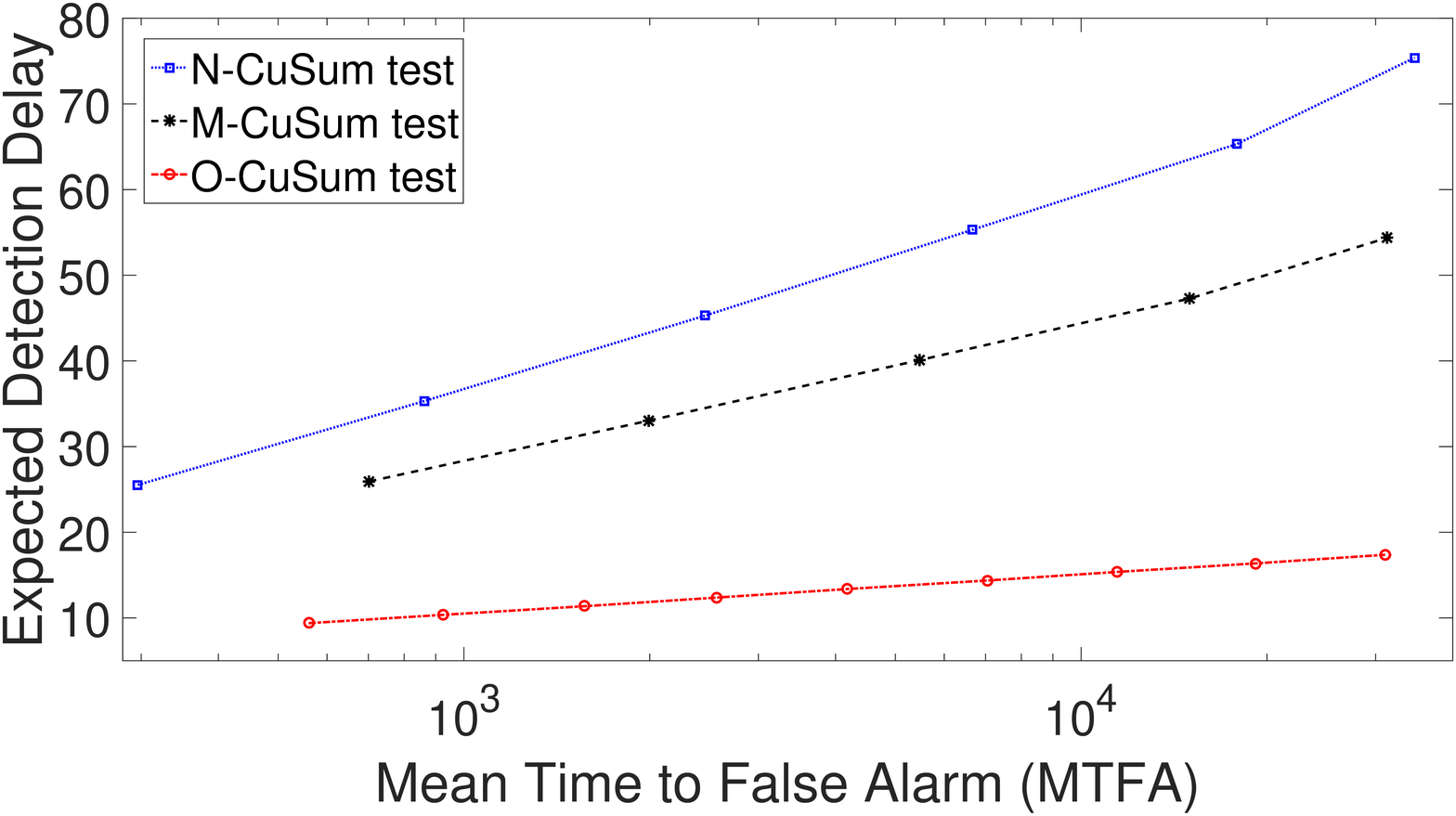,width=0.4\textwidth}}
}
\hspace{0.1in}
\subfigure[$\mathrm{WADD}$ versus MTFA for $L=10$, $m=1$.]{
\label{fig:L-10}
\mbox{
\epsfig{file=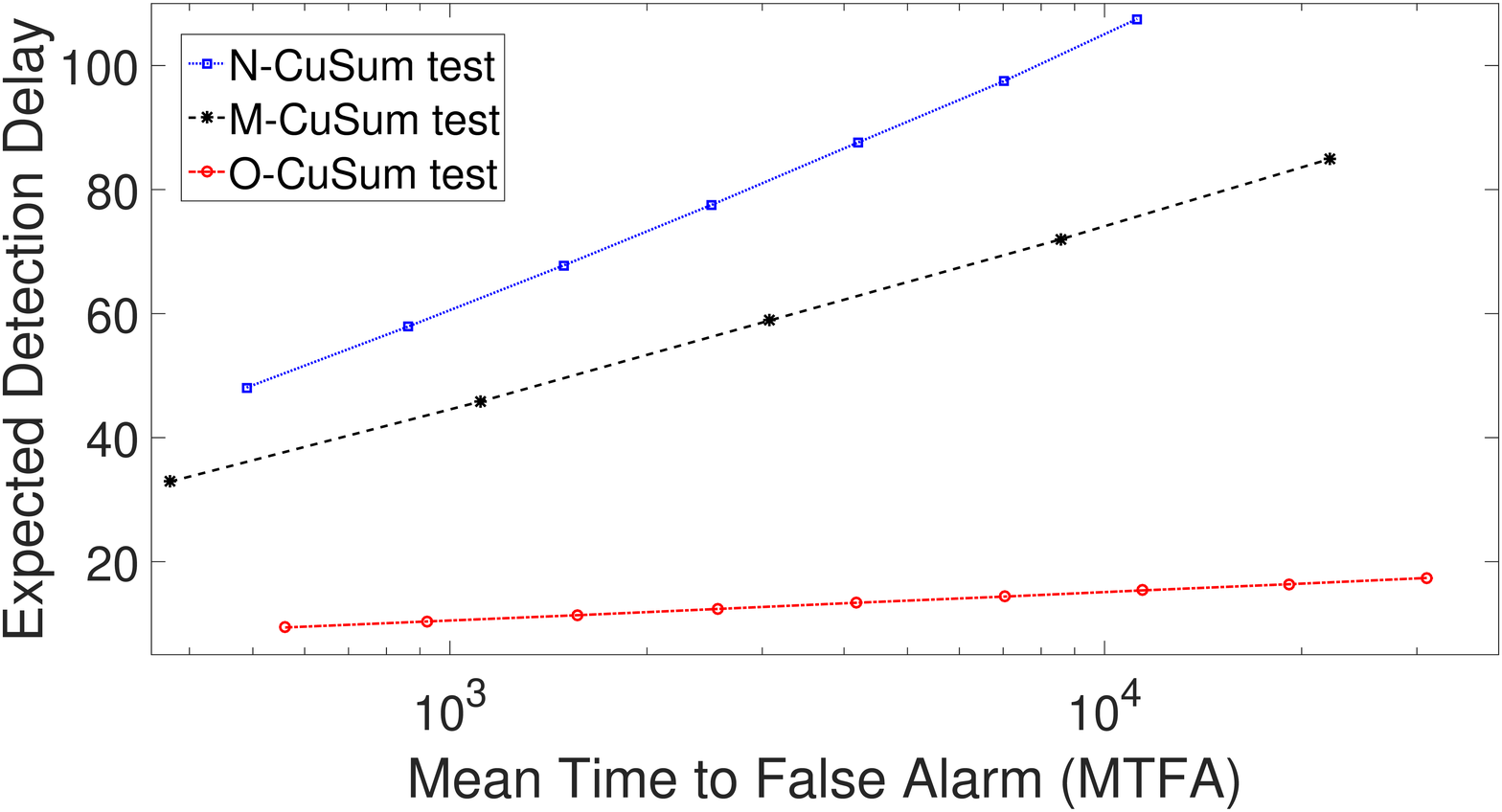,width=0.4\textwidth}}
}
\caption{$\mathrm{WADD}$ versus MTFA for homogeneous sensor network.}
\label{fig:FIGS1}
\end{figure*}

\begin{figure*}[t]
\centering
\subfigure[$\mathrm{WADD}$ versus MTFA for $L=20$, $m=1$.]{
\label{fig:L-20}
\mbox{
\epsfig{file=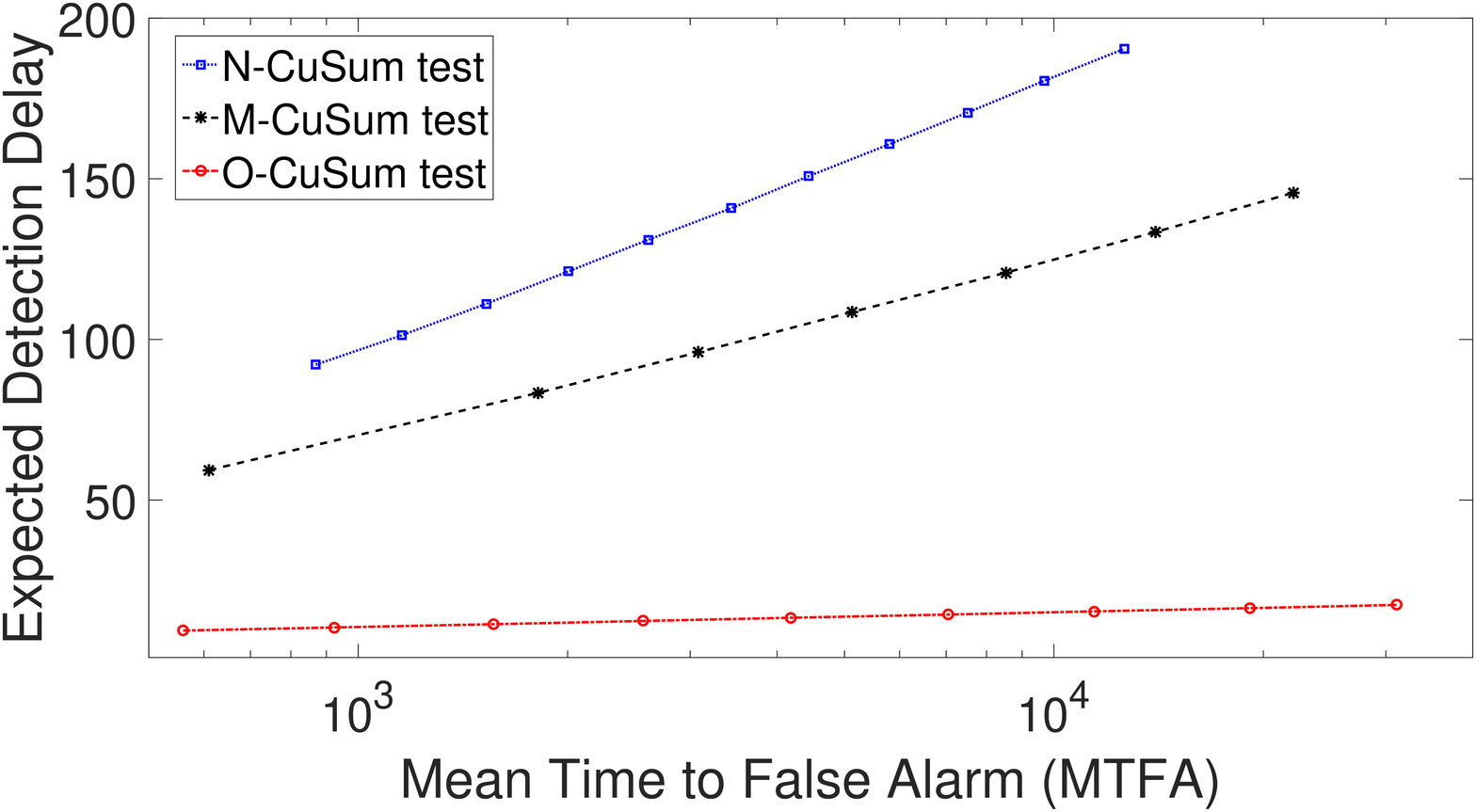,width=0.4\textwidth}}
}
\hspace{0.1in}
\subfigure[$\mathrm{WADD}$ versus MTFA for the M-CuSum when $m =1$ and for different $L$ values.]{
\label{fig:different_L_comparison}
\mbox{
\epsfig{file=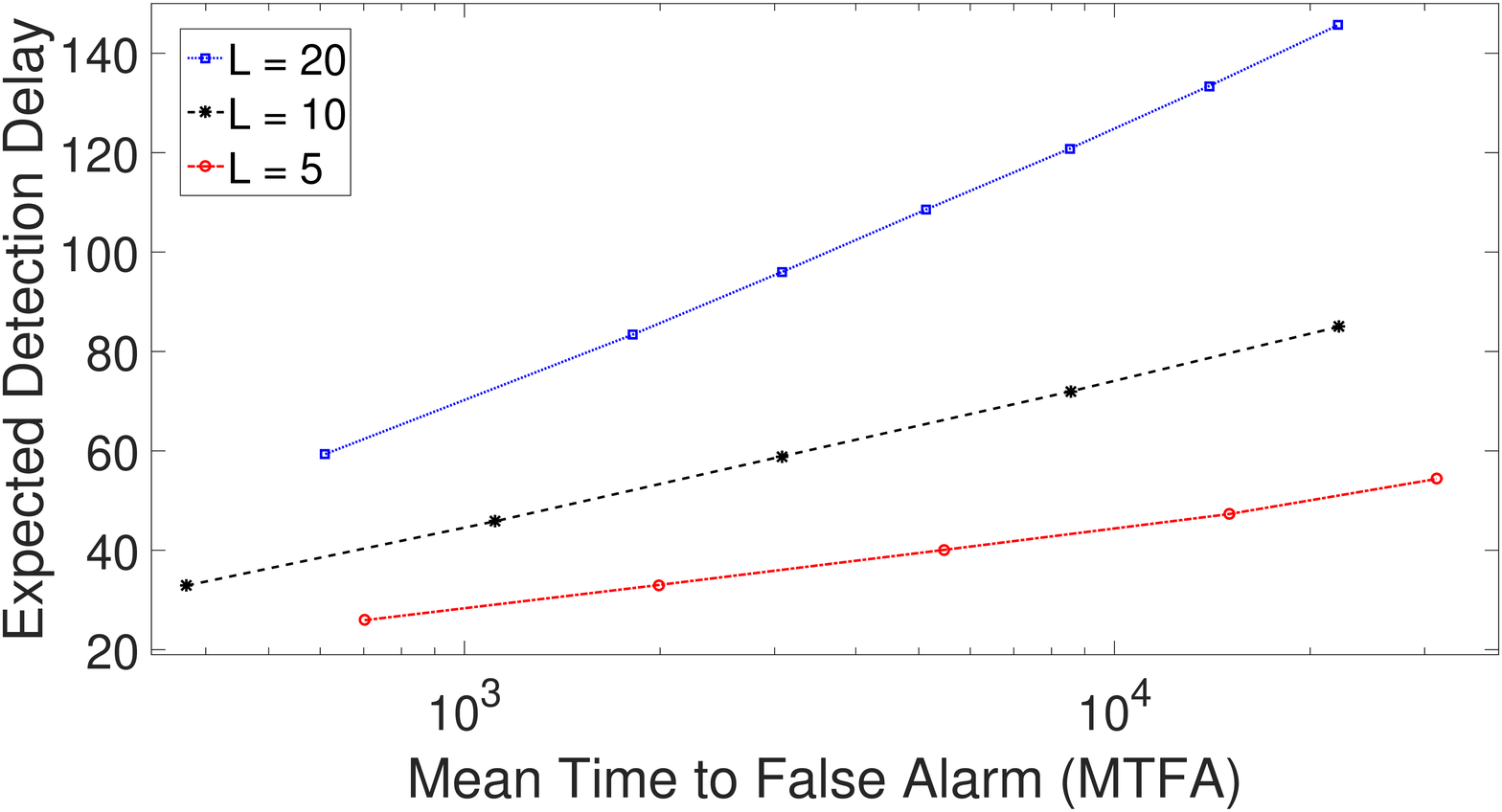,width=0.4\textwidth}}
}
\caption{$\mathrm{WADD}$ versus MTFA for homogeneous sensor network.}
\label{fig:FIGS2}
\end{figure*}

\begin{figure*}[t]
\centering
\subfigure[$\mathrm{WADD}$ versus MTFA for $L=10$, $m=1$.]{
\label{fig:hetero_1}
\mbox{
\epsfig{file=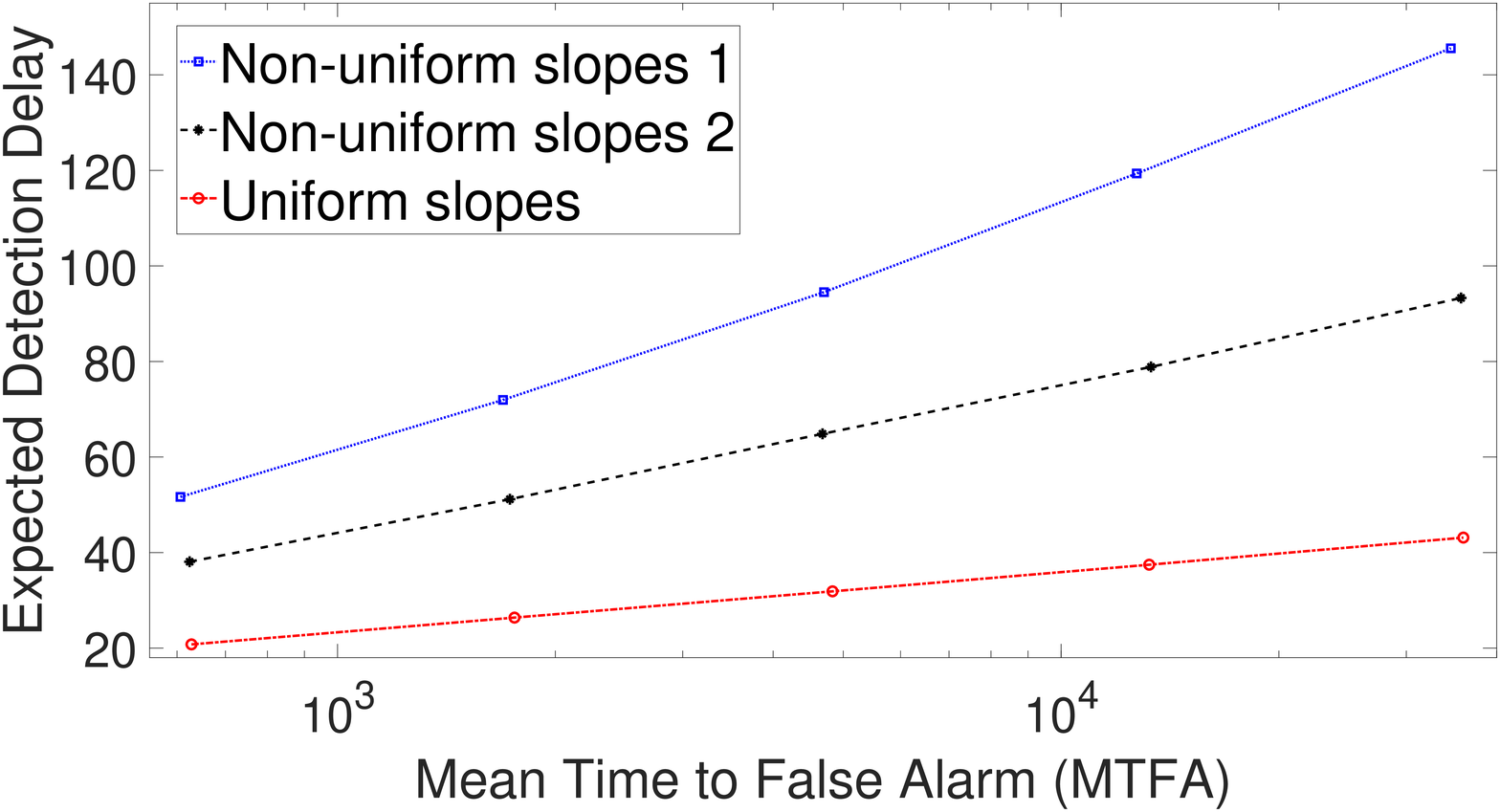,width=0.4\textwidth}}
}
\hspace{0.1in}
\subfigure[$\mathrm{WADD}$ versus MTFA for $L=20$, $m=1$.]{
\label{fig:hetero_2}
\mbox{
\epsfig{file=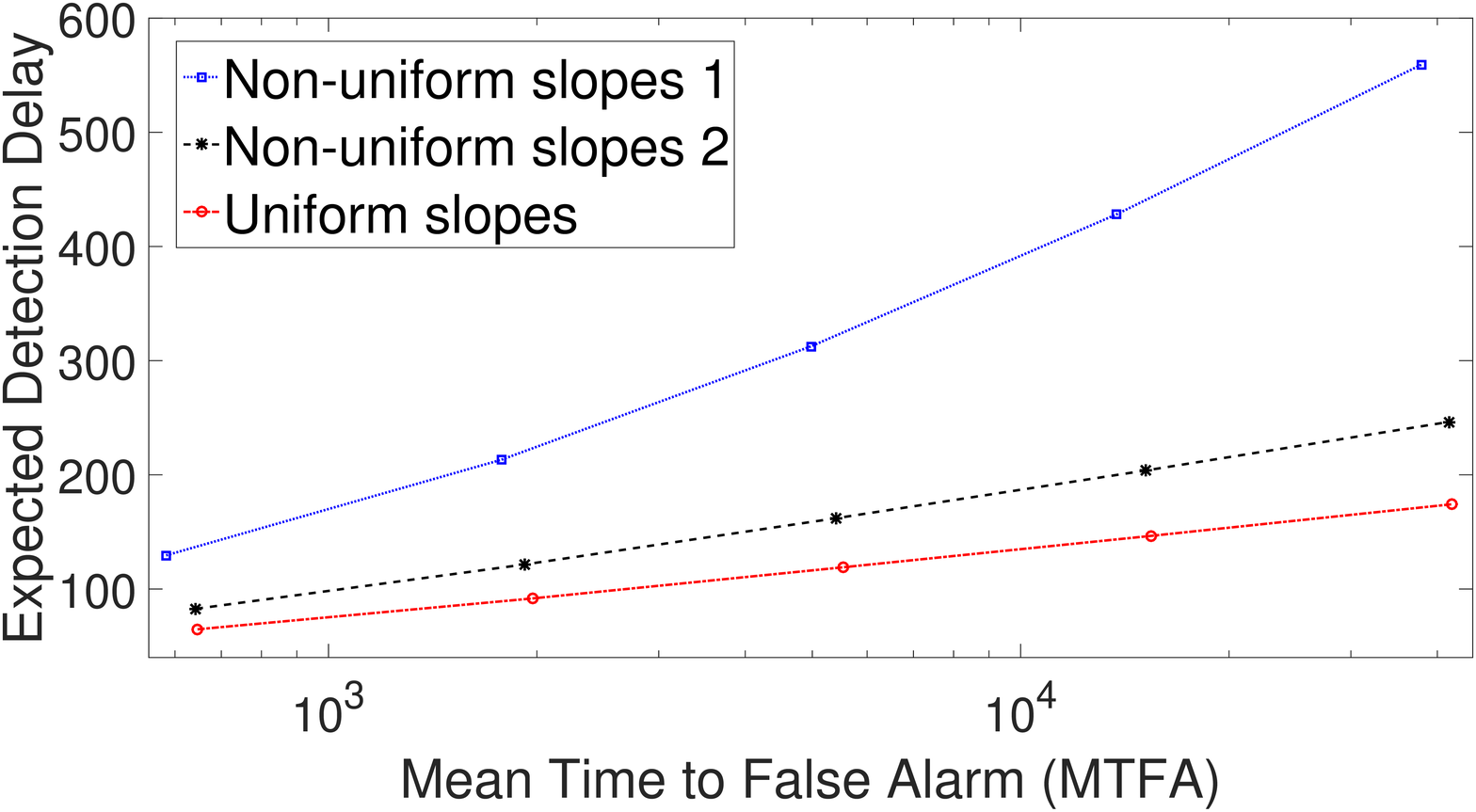,width=0.4\textwidth}}
}
\caption{$\mathrm{WADD}$ versus MTFA for heterogeneous sensor network.}
\label{fig:FIGS3}
\end{figure*}
For the case of a heterogeneous sensor network, we compare three versions of the test introduced in eqs. \eqref{eq:test_stat} - \eqref{eq:test_rec}: the first version (``Uniform slopes'' in Fig. \ref{fig:FIGS3}) uses the optimal weights $\bm{\alpha}^*$ to achieve a uniform average statistic drift among anomaly placements (see Lemma \ref{equal_drifts_lemma}); the second and third versions (``Non-uniform slopes 1'' and ``Non-uniform slopes 2'' in Figs. \ref{fig:FIGS3}) use arbitrary choices of weights that only guarantee that the expected drift of the statistic is positive for any placement of the anomaly. The optimal weights are found by using gradient descent with the derivatives calculated as in eq. \eqref{eq:eksiswsoyla}. Note that each derivative depends is equal to a difference of two expected values, which we calculate through Monte Carlo. Furthermore, It should be noted that the $\mathrm{WADD}$ in the case of heterogeneous sensor networks is calculated approximately, since the worst path of the anomaly cannot be specified analytically. However, as the MTFA becomes large, $\mathrm{WADD}$ can be approximated by placing the anomalies at only the nodes (in this case node since $m=1$) that correspond to the worst post-change expected drift. For the optimal weight choice, the placement of the anomaly does not affect the delay for large MTFA, since the expected drift does not depend on the trajectory of the anomaly. 

We consider the cases of $L=10$ and $L=20$. For the case of $L=10$, we assume that $g_\ell=\mathcal{N}(0,1)$ for all $\ell \in [L]$, and that $f_\ell = \mathcal{N}(\mu_\ell,1)$ with $\bm{\mu} = [1, 1.1, 1.2, 1.3, 1.4, 1.5, 1.6, 1.7, 1.8, 1.9]^\top$ denoting the vector of the anomalous means. The results can be seen in Fig. \ref{fig:hetero_1}. The M-CUSUM test statistic using optimal weights is then characterized by a uniform average statistic drift, approximately equal to $0.178$. For the case of ``Non-uniform slopes 1'' the worst expected drift corresponds to placing the anomaly at sensor 2, corresponding to an approximate slope of $0.029$, and for the case of ``Non-uniform slopes 2'' at sensor 5, with an approximate slope of $0.065$. We see that he mixture-CUSUM test using the optimal weights $\bm{\alpha}^*$ outperforms the other two implementations.  Similar results can be produced by considering the case of $L=20$. For that case, we assume that $g_\ell=\mathcal{N}(0,1)$ for all $\ell \in [L]$, $f_\ell=\mathcal{N}(0.8,1)$ for all $1 \leq \ell \leq 5 $, $f_\ell=\mathcal{N}(1,1)$ for all $6 \leq \ell \leq 15 $, and $f_\ell=\mathcal{N}(1.2,1)$ for all $16 \leq \ell \leq 20 $. The results can be seen in Fig. \ref{fig:hetero_2}, where we note that the optimal weights test outperforms the tests that use arbitrarily chosen weights. The resulting homogeneous average statistic drift is then approximately equal to $0.036$. Furthermore, for the case of ``Non-uniform slopes 1'' the worst expected drift corresponds to placing the anomaly at any sensor $\ell \in [5]$, corresponding to an approximate slope of $0.003$, and for the case of ``Non-uniform slopes 2'' at any sensor $ \ell \in \{16,17,18,19,20\}$, with an approximate slope equal to $0.023$. Finally, it should be noted that in this case we have chosen ``Non-uniform slopes 1'' to correspond to the case of uniform weights. As a result, the gap between the blue and red lines in Fig. \ref{fig:hetero_2} captures the loss we suffer if we make the assumption that the sensors of the network are homogeneous.
\section{Conclusion}
\label{sec:conc}
In this paper, we studied the problem of moving anomaly detection, where an anomaly evolves around a sensor network affecting different nodes at each time instant after its appearance. We posed the problem into a minimax QCD setting, where the trajectory of the anomaly is supposed to be \textit{unknown} but \textit{deterministic}. To this end, we introduced a modified version of Lorden's \cite{lorden:1971} detection delay metric that evaluates candidate detection schemes according to the worst performance with respect to the path of the anomaly. We proposed a CUSUM-type test that is an exact solution to the moving anomaly QCD problem for the case of a homogeneous network, and is also first-order asymptotically optimal when applied to a heterogeneous network. Due to the lack of a specific anomaly evolution model and the use of a worst-path delay approach, our exact test structure depends on the structure of the network, in particular the data-generating pdfs at each sensor.

Future work in this area includes studying the case of a moving anomaly of size varying with time (for current progress in this problem see \cite{Rovatsos_ICASSP:2020}),  modifying proposed procedures to provide robustness with respect to limited knowledge of data-generating distributions, as well as, studying the case of moving anomaly detection under the presence of an adversary.

\appendix
\begin{lemma}
\label{thm:lemma_trunc}
For any stopping time $\tau$ adapted to $\mathscr{F}$ and $N>0$ define the truncated version of $\tau$ by $\tau^{(N)} \triangleq \min\{\tau, N\}$. We then have that
\begin{align}
\label{eq:math_4}
\mathrm{WADD}(\tau^{(N)}) \leq \mathrm{WADD}(\tau).
\end{align}
\end{lemma}
\begin{proof}
Fix $\nu \geq 0$. Consider initially that $N > \nu$. Then, since $\{\tau^{(N)} > \nu\} = \{\min\{\tau,N\} > \nu \} = \{\tau > \nu\} \cap \{N> \nu\}$, we have that $\{\tau^{(N)}> \nu \} = \{\tau> \nu\}.$ Since $\tau^{(N)} \leq \tau$, this implies that for any $N > \nu$ and any $\bm{S}$ we have that
\begin{align}
\label{eq:math_1}
&\mathbb{E}_\nu^{\bm{S}} \left[\tau^{(N)}-\nu |\tau_N> \nu , \mathscr{F}_{\nu  } \right]
=\mathbb{E}_\nu^{\bm{S}} \left[\tau^{(N)}-\nu  |\tau >\nu , \mathscr{F}_{\nu } \right]
\leq \mathbb{E}_\nu^{\bm{S}} \left[\tau-\nu  |\tau>\nu , \mathscr{F}_{\nu  } \right].
\end{align}
For the case of $N\leq\nu$, we have that that $\mathbb{P}_\nu^{\bm{S}}(\tau^{(N)} > \nu)=0$, which implies that by convention for any $N \leq \nu$ and any $\bm{S}$ we have that
\begin{align}
\label{eq:math_2}
\mathbb{E}_\nu^{\bm{S}} \left[\tau^{(N)}-\nu  |\tau^{(N)} > \nu , \mathscr{F}_{\nu  } \right]=1.
\end{align}
Furthermore, note that for any $\bm{S}$ we have that
\begin{align}
\label{eq:math_3}
\mathbb{E}_\nu^{\bm{S}} \left[\tau -\nu  |\tau > \nu , \mathscr{F}_{\nu} \right]\geq 1.
\end{align}
From \eqref{eq:math_1} - \eqref{eq:math_3} we have that for any $\nu \geq 0$ and  any $\bm{S}$
\begin{align}
\mathbb{E}_\nu^{\bm{S}} \left[\tau^{(N)} -\nu  |\tau^{(N)} > \nu , \mathscr{F}_{\nu} \right] \leq \mathbb{E}_\nu^{\bm{S}} \left[\tau -\nu  |\tau > \nu , \mathscr{F}_{\nu} \right].
\end{align}
By taking the sup and ess sup on both sides the lemma is established.
\end{proof}
\begin{lemma}
\label{converge}
Let $A>0$, $\tau$ a stopping time adapted to $\mathscr{F}$ such that $\mathbb{E}_\infty[\tau] < \infty$, and $\Phi : \mathbb{R} \mapsto \mathbb{R}$ a function satisfying $|\Phi(x)| \leq A \text{ for all } x \in \mathbb{R}$. Then for any $\bm{\lambda} \in \mathcal{A}$ we have that
\begin{align}
\lim\limits_{N\rightarrow \infty} \mathbb{E}_\infty \left[ \sum\limits_{k=0}^{\tau^{(N)}-1} \Phi(W_{\bm{\lambda}}[k])\right] =  \mathbb{E}_\infty \left[ \sum\limits_{k=0}^{\tau-1} \Phi(W_{\bm{\lambda}}[k])\right].
\end{align}
\end{lemma}
\begin{proof}
Note that since $\tau \geq \tau^{(N)}$ we have that
\begin{align}
\label{eq:leme3}
\mathbb{E}_\infty \left[ \sum\limits_{k=0}^{\tau-1} \Phi(W_{\bm{\lambda}}[k])\right] = \mathbb{E}_\infty \left[ \sum\limits_{k=0}^{\tau^{(N)}-1} \Phi(W_{\bm{\lambda}}[k])\right]+\mathbb{E}_\infty \left[ \sum\limits_{k=\tau^{(N)}}^{\tau-1} \Phi(W_{\bm{\lambda}}[k])\right].
\end{align}
Furthermore, note that by using Jensen's and triangle inequalities together with the assumption that $\Phi(x)$ is bounded we have that
\begin{align}
\label{eq:lemeq}
&\mathbb{E}_\infty \left[ \sum\limits_{k=\tau^{(N)}}^{\tau-1} \Phi(W_{\bm{\lambda}}[k])\right]  \leq\bigg|\mathbb{E}_\infty \left[ \sum\limits_{k=\tau^{(N)}}^{\tau-1} \Phi(W_{\bm{\lambda}}[k])\right]\bigg|  \leq \mathbb{E}_\infty \left[ \sum\limits_{k=\tau^{(N)}}^{\tau-1} \bigg|\Phi(W_{\bm{\lambda}}[k])\bigg|\right]\nonumber \\&\leq A \mathbb{E}_\infty[ \tau-\tau^{(N)}] = A\mathbb{E}_\infty[(\tau-N)^+].
\end{align}
By properties of the expectation of positive random variables, we then note that 
\begin{align}
&\mathbb{E}_\infty[(\tau-N)^+]= \sum\limits_{j=0}^\infty \mathbb{P}_\infty((\tau-N)^+ >j) \nonumber \\&= \sum\limits_{j=0}^\infty\mathbb{P}_\infty(\tau>j+N) = \sum\limits_{j=N}^\infty\mathbb{P}_\infty(\tau>j) 
\end{align}
which since, by assumption, $\mathbb{E}_\infty[\tau] = \sum\limits_{j=0}^\infty\mathbb{P}_\infty(\tau>j) <\infty$ implies that
\begin{align}
\lim_{N\rightarrow\infty}\mathbb{E}_\infty[(\tau-N)^+]=\lim_{N\rightarrow\infty}\mathbb{P}_\infty(\tau>N) =0.
\end{align}
As a result, from \eqref{eq:lemeq} we have that 
\begin{align}
\label{eq:lemeq_2}
\lim_{N\rightarrow\infty}\mathbb{E}_\infty \left[ \sum\limits_{k=\tau^{(N)}}^{\tau-1} \Phi_{\bm{\lambda}}(W[k])\right] =0.
\end{align}
After taking the limit in both sides of \eqref{eq:leme3} and using eq. \eqref{eq:lemeq_2} the lemma is established.
\end{proof}
\begin{proof}[Proof of Theorem 1]
Fix $\bm{\alpha} \in \mathcal{A}$. Due to the presence of the sup and ess sup in \eqref{eq:delay_metric}, we have that for any path $\bm{S}$, $\nu \geq 0$, $\mathscr{F}$-adapted stopping time $\tau$ and $N>0$
\begin{align}
\label{eq:eq1}
\mathrm{WADD}(\tau^{(N)}) &\geq \mathbb{E}_\nu^{\bm{S}} \left[ \tau^{(N)}-\nu | \tau^{(N)}> \nu , \mathscr{F}_{\nu} \right] 
= \mathbb{E}_\nu^{\bm{S}} \left[ \sum_{j=\nu}^\infty \mathbbm{1}_{\{\tau^{(N)} > j\}} \bigg|\tau^{(N)} > \nu, \mathscr{F}_{\nu}  \right] \nonumber \\& \stackrel{(a)}{=}  \mathbb{E}_\infty \left[ \sum_{j=\nu}^\infty\Gamma_{\bm{S}}(j,\nu)  \mathbbm{1}_{\{\tau^{(N)}> j\}} \bigg|\tau^{(N)} > \nu, \mathscr{F}_{\nu}  \right]
\end{align}
where $(a)$ follows by changing the measure to $\mathbb{P}_\infty(\cdot)$. By multiplying both sides of the inequality \eqref{eq:eq1} with $\mathbbm{1}_{ \{ \tau^{(N)} > \nu \} }(1-W_{\bm \alpha}[\nu])^+$ and taking the expected value under $\mathbb{E}_\infty[\cdot]$ we have that
\begin{align}
\label{eq:eq2}
&\mathbb{E}_\infty \nonumber \left[\mathbbm{1}_{ \{ \tau^{(N)} > \nu \} }(1-W_{\bm \alpha}[\nu])^+\mathrm{WADD}(\tau^{(N)})\right] 
\nonumber
\\ \nonumber &\geq \mathbb{E}_\infty\left[ \mathbbm{1}_{ \{ \tau^{(N)}> \nu \} }(1-W_{\bm \alpha}[\nu])^+\mathbb{E}_\infty \left[ \sum_{j=\nu}^\infty\Gamma_{\bm{S}}(j,\nu)  \mathbbm{1}_{\{\tau^{(N)}> j\}} \bigg| \tau^{(N)} > \nu, \mathscr{F}_{\nu} \right]\right] 
\\& 
\stackrel{(b)}{=}  \mathbb{E}_\infty  \left[ \mathbb{E}_\infty  \left[\mathbbm{1}_{ \{ \tau^{(N)}> \nu \} }(1-W_{\bm \alpha}[\nu])^+\sum_{j=\nu}^\infty\Gamma_{\bm{S}}(j,\nu)  \mathbbm{1}_{\{\tau^{(N)}>j\}}\bigg| \tau^{(N)} > \nu, \mathscr{F}_{\nu}\right]\right]\nonumber 
\\& 
\stackrel{(c)}{=}  \mathbb{E}_\infty \left[ \sum_{j=\nu}^\infty\mathbbm{1}_{ \{ \tau^{(N)} >\nu \} }(1-W_{\bm \alpha}[\nu])^+\Gamma_{\bm{S}}(j,\nu)  \mathbbm{1}_{\{\tau^{(N)} > j\}}\right].
\end{align}
where $(b)$ follows since $\mathbbm{1}_{ \{ \tau^{(N)}> \nu \} }(1-W_{\bm \alpha}[\nu])^+$ is $\mathscr{F}_\nu$-measurable and, hence, can go inside the expectation since the conditioning is with respect to $\mathscr{F}_\nu$, and (c) follows from the tower property of expectations. 
By summing over $\nu$ from $\nu =0$ to $\nu = N$, and due to the linearity of expectation and the fact that $\tau^{(N)} \leq N$ we have that
\begin{align}
&\mathbb{E}_\infty \nonumber \left[\sum\limits_{\nu=0}^{\tau^{(N)}-1}\mathbbm{1}_{ \{ \tau^{(N)} > \nu \} }(1-W_{\bm \alpha}[\nu])^+\mathrm{WADD}(\tau^{(N)})\right] 
\nonumber \\& 
 \geq 
 \mathbb{E}_\infty \left[ \sum\limits_{\nu=0}^{\tau^{(N) }-1}\sum_{j=\nu}^\infty\mathbbm{1}_{ \{ \tau^{(N)} > \nu \} }(1-W_{\bm \alpha}[\nu])^+\Gamma_{\bm{S}}(j,\nu)  \mathbbm{1}_{\{\tau^{(N)} > j\}}\right],
\end{align}
which in turn, implies that
\begin{align}
\nonumber &\mathbb{E}_\infty \nonumber \left[\sum\limits_{\nu=0}^{\tau^{(N)} - 1}(1-W_{\bm \alpha}[\nu])^+\mathrm{WADD}(\tau^{(N)})\right] 
\geq \mathbb{E}_\infty \left[ \sum\limits_{\nu=0}^{\tau^{(N)}-1}\sum_{j=\nu}^{\tau^{(N)}-1}(1-W_{\bm \alpha}[\nu])^+\Gamma_{\bm{S}}(j,\nu)  \right]
\nonumber 
\\& \stackrel{(d)}{=} \mathbb{E}_\infty \left[ \sum\limits_{j=0}^{\tau^{(N)}-1}\sum\limits_{\nu=0}^{j}(1-W_{\bm \alpha}[\nu])^+\Gamma_{\bm{S}}(j,\nu)  \right]
\end{align}
where $(d)$ follows after changing the order of the summation. Since $\mathrm{WADD}(\tau^{(N)})$ is a constant, and therefore can go outside of the expectation, we then have that 
\begin{align}
\label{eq:eq5} 
 &\mathrm{WADD}(\tau^{(N)}) \geq \frac{ \mathbb{E}_\infty \left[ \sum\limits_{j=0}^{\tau^{(N)}-1}\sum\limits_{\nu=0}^{j}(1-W_{\bm \alpha}[\nu])^+\Gamma_{\bm{S}}(j,\nu)  \right]}{\mathbb{E}_\infty  \left[\sum\limits_{\nu=0}^{\tau^{(N)}-1}(1-W_{\bm \alpha}[\nu])^+\right] } .
\end{align}
By taking the sup with respect to $\bm{S}$, and since the right hand side fraction depends on $\bm{S}$ only through $\bm{S}[1, N-1]$ we have that
\begin{align}
 &\mathrm{WADD}(\tau^{(N)}) \geq \sup\limits_{\bm{S}[1, N-1]}\frac{ \mathbb{E}_\infty \left[ \sum\limits_{j=0}^{\tau^{(N)}-1}\sum\limits_{\nu=0}^{j}(1-W_{\bm \alpha}[\nu])^+\Gamma_{\bm{S}}(j,\nu)  \right]}{\mathbb{E}_\infty  \left[\sum\limits_{\nu=0}^{\tau^{(N)}-1}(1-W_{\bm \alpha}[\nu])^+\right] }.
\end{align}
Since the denominator in the right hand side does not depend on $\bm{S}$, we have that
\begin{align}
\label{eq:eq5b} 
 &\mathrm{WADD}(\tau^{(N)}) \geq \frac{\sup\limits_{\bm{S}[1, N-1]} \mathbb{E}_\infty \left[ \sum\limits_{j=0}^{\tau^{(N)}-1}\sum\limits_{\nu=0}^{j}(1-W_{\bm \alpha}[\nu])^+\Gamma_{\bm{S}}(j,\nu)  \right]}{\mathbb{E}_\infty  \left[\sum\limits_{\nu=0}^{\tau^{(N)}-1}(1-W_{\bm \alpha}[\nu])^+\right] }.
\end{align}
To proceed, we further bound the numerator in \eqref{eq:eq5b}. For $1 \leq n < N$, define the following function
 \begin{align}
 \label{eq:def_phi}
 &\Phi_{n,N-1}(\bm{S}[1,{n-1}],\bm{S}[n+1,N-1]) \triangleq \sup\limits_{\bm{S}[n] } \mathbb{E}_\infty \left[ \sum\limits_{j=0}^{N-1}\sum\limits_{\nu=0}^{j}(1-W_{\bm \alpha}[\nu])^+\Gamma_{\bm{S}}(j,\nu) \mathbbm{1}_{\{\tau^{(N)} > j\}} \right] .
 \end{align}
 Then, by first taking the sup over $\bm{S}[n]$ we have that
\begin{align}
\label{eq:eq7}
&\sup\limits_{\bm{S}[1,N-1]}  \mathbb{E}_\infty \left[ \sum\limits_{j=0}^{N-1}\sum\limits_{\nu=0}^{j}(1-W_{\bm \alpha}[\nu])^+\Gamma_{\bm{S}}(j,\nu) \mathbbm{1}_{\{\tau^{(N)} > j\}} \right]
\nonumber\\ \nonumber & = \sup\limits_{\bm{S}[1,n-1] ,\bm{S}[n+1,N-1]} \left[\sup\limits_{\bm{S}[n]}   \mathbb{E}_\infty \left[ \sum\limits_{j=0}^{N-1}\sum\limits_{\nu=0}^{j}(1-W_{\bm \alpha}[\nu])^+\Gamma_{\bm{S}}(j,\nu) \mathbbm{1}_{\{\tau^{(N)} > j\}} \right]\right] \nonumber 
 \nonumber  \\& =
   \sup\limits_{\bm{S}[1,n-1] ,\bm{S}[n+1,N-1]} \Phi_{n,N-1}(\bm{S}[1,{n-1}],\bm{S}[n+1,N-1]).
 \end{align}
Note that under $\mathbb{P}_\infty(\cdot)$ and for $j$ such that $0 \leq j < n < N$ we have that 
  \begin{align}
\label{eq:eq8}
\sum\limits_{\nu=0}^{j}(1-W_{\bm \alpha}[\nu])^+\Gamma_{\bm{S}}(j,\nu) \mathbbm{1}_{\{\tau^{(N)} > j\}}
\end{align}
is independent of $\bm{S}[n]$. For $0 \leq n \leq j < N$ we have that 
  \begin{align}
\label{eq:eq9}
&\sum\limits_{\nu=0}^{j}(1-W_{\bm \alpha}[\nu])^+\Gamma_{\bm{S}}(j,\nu) \mathbbm{1}_{\{\tau^{(N)} > j\}}\nonumber
 =\sum\limits_{\nu=0}^{n-1}(1-W_{\bm \alpha}[\nu])^+\Gamma_{\bm{S}}(j,\nu) \mathbbm{1}_{\{\tau^{(N)} > j\}} 
 \\& \nonumber
 + \sum\limits_{\nu=n}^{j}(1-W_{\bm \alpha}[\nu])^+\Gamma_{\bm{S}}(j,\nu) \mathbbm{1}_{\{\tau^{(N)} > j\}}
 \nonumber \\&=
\Gamma_{\bm{S}}(n,n-1)\left(\sum_{\nu=0}^{n-1}(1-W_{\bm \alpha}[\nu])^+ \left(\prod\limits_{\substack{i=\nu+1\\i\neq n}}^{j}\Gamma_{\bm{S}}(i,i-1)\right)  \mathbbm{1}_{\{\tau^{(N)} > j\}} \right)
\nonumber\\&+ \sum\limits_{\nu=n}^{j}(1-W_{\bm \alpha}[\nu])^+\Gamma_{\bm{S}}(j,\nu) \mathbbm{1}_{\{\tau^{(N)} > j\}},
\end{align}
where under $\mathbb{P}_\infty(\cdot)$ the dependence from $\bm{S}[n]$ is only through the likelihood ratio $\Gamma_{\bm{S}}(n,n-1)$ of the first term. 

For $0 \leq j < N$ and $0 \leq n < N$ define
\begin{align}
\label{eq:def1}
A_{j,n} = \left(\sum_{\nu=0}^{n-1}(1-W_{\bm \alpha}[\nu])^+ \left(\prod\limits_{\substack{i=\nu+1\\i\neq n}}^{j}\Gamma_{\bm{S}}(i,i-1)\right)  \mathbbm{1}_{\{\tau^{(N)} > j\}} \right)\mathbbm{1}_{\{j\geq n\}}
\end{align}
and
\begin{align}
\label{eq:def2}
B_{j,n} &\triangleq  \left(\sum\limits_{\nu=0}^{j}(1-W_{\bm \alpha}[\nu])^+\Gamma_{\bm{S}}(j,\nu) \mathbbm{1}_{\{\tau^{(N)} > j\}} \right)\mathbbm{1}_{\{j < n\}} \nonumber \\& +  \bigg(\sum\limits_{\nu=n}^{j}(1-W_{\bm \alpha}[\nu])^+\Gamma_{\bm{S}}(j,\nu) \mathbbm{1}_{\{\tau^{(N)} > j\}} \bigg)\mathbbm{1}_{\{j\geq n\}}.
\end{align}

As a result, from eqs. \eqref{eq:eq9} - \eqref{eq:def2} we have that for any $0\leq n < N$
\begin{align}
\label{eq:conc}
 &\sum\limits_{\nu=0}^{j}(1-W_{\bm \alpha}[\nu])^+\Gamma_{\bm{S}}(j,\nu) \mathbbm{1}_{\{\tau^{(N)} > j\}}
= \Gamma_{\bm{S}}(n,n-1)A_{j,n} + B_{j,n}.
\end{align}
Then from eqs. \eqref{eq:def_phi}, \eqref{eq:conc} we have that
\begin{align}
\label{eq:eq10}
&\Phi_{n,N-1}(\bm{S}[1,{n-1}],\bm{S}[n+1,N-1]) =\sup\limits_{\bm{S}[n]}   \mathbb{E}_\infty \bigg[ \sum_{j=0}^{N-1} \bigg(\Gamma_{\bm{S}}(n,n-1) A_{j,n}+ B_{j,n} \bigg)\bigg] \nonumber \\&= \sup\limits_{\bm{S}[n]}   \mathbb{E}_\infty \bigg[  \Gamma_{\bm{S}}(n,n-1)  \sum_{j=0}^{N-1} A_{j,n} + \sum_{j=0}^{N-1}B_{j,n} \bigg].
\end{align}
Note that since $A_{j,n}$ and $B_{j,n}$ are independent of $\bm{S}[n]$ under $\mathbb{P}_\infty(\cdot)$, we have that for all $\bm E \in \mathcal{E}$
\begin{align}
\label{eq:eq11}
& \sup\limits_{\bm{S}[n]}   \mathbb{E}_\infty \bigg[  \Gamma_{\bm{S}}(n,n-1)  \sum_{j=0}^{N-1} A_{j,n} + \sum_{j=0}^{N-1}B_{j,n} \bigg] 
= \sup\limits_{\bm{S}[n]}   \mathbb{E}_\infty \left[ \left( \prod\limits_{\ell \, \in \, \bm{S}[n]} \frac{f_\ell(X_{\ell}[n])}{g_\ell(X_{\ell}[n])}   \right) \sum_{j=0}^{N-1} A_{j,n} + \sum_{j=0}^{N-1}B_{j,n} \right]  \nonumber \\& 
\geq    \mathbb{E}_\infty \left[ \left( \prod\limits_{\ell \, \in \, \bm{E}} \frac{f_\ell(X_{\ell}[n])}{g_\ell(X_{\ell}[n])}   \right) \sum_{j=0}^{N-1} A_{j,n} + \sum_{j=0}^{N-1}B_{j,n} \right],
\end{align}
which together with eq. \eqref{eq:eq10} implies that
\begin{align}
\label{eq:eq11b}
\Phi_{n,N-1}(\bm{S}[1,{n-1}],\bm{S}[n+1,N-1]) \geq    \mathbb{E}_\infty \left[ \left( \prod\limits_{\ell \, \in \, \bm{E}} \frac{f_\ell(X_{\ell}[n])}{g_\ell(X_{\ell}[n])}   \right) \sum_{j=0}^{N-1} A_{j,n} + \sum_{j=0}^{N-1}B_{j,n} \right].
\end{align}
By averaging both sides of eq. \eqref{eq:eq11b} with respect to $\bm{\alpha}$ we then have that
\begin{align}
\label{eq:yoeqeq}
&\Phi_{n,N-1}(\bm{S}[1,{n-1}],\bm{S}[n+1,N-1]) \nonumber =\sum\limits_{\bm{E} \,\in\, \mathcal{E}} \alpha_{\bm E}\Phi_{n,N-1}(\bm{S}[1,{n-1}],\bm{S}[n+1,N-1])
\nonumber \\&\geq   
\sum\limits_{\bm{E} \,\in\, \mathcal{E}} \bm{\alpha}_{\bm E}\mathbb{E}_\infty \left[ \left( \prod\limits_{\ell \, \in \, \bm{E}} \frac{f(X_{\ell}[n])}{g(X_{\ell}[n])}   \right) \sum_{j=0}^{N-1} A_{j,n} + \sum_{j=0}^{N-1}B_{j,n} \right]  \nonumber
\\&
=\mathbb{E}_\infty \left[ \left(\sum\limits_{\bm{E} \,\in\, \mathcal{E} }\alpha_{\bm E}\left( \prod\limits_{\ell \, \in \, \bm{E}} \frac{f_\ell(X_{\ell}[n])}{g_\ell(X_{\ell}[n])}   \right) \right)\sum_{j=1}^{N} A_{j,n} + \sum_{j=1}^{N}B_{j,n} \right] 
\nonumber \\&=
\mathbb{E}_\infty \left[\mathcal{L}_{\bm \alpha}(n,n-1) \left( \sum_{j=0}^{N-1} A_{j,n}\right) + \sum_{j=0}^{N-1}B_{j,n} \right] 
\nonumber \\& 
 = \mathbb{E}_\infty \bigg[ \sum\limits_{j=0}^{N-1}\sum\limits_{\nu=0}^{j}(1-W_{\bm \alpha}[\nu])^+
 \mathcal{L}_{\bm \alpha}(n,n-1)\left(\prod\limits_{\substack{i=\nu+1\\i\neq n}}^{j-1}
 \Gamma_{\bm{S}}(i,i-1)\right)  \mathbbm{1}_{\{\tau^{(N)} > j\}} \bigg].
\end{align}
By unfolding eq. \eqref{eq:eq7} in the same fashion with respect to all $0 \leq n < N$, it can be easily shown that

\begin{align}
\label{eq:eq14}
&\sup\limits_{\bm{S}[1,N-1]}  \mathbb{E}_\infty \left[ \sum\limits_{j=0}^{\tau^{(N)} - 1}\sum\limits_{\nu=0}^{j}(1-W_{\bm \alpha}[\nu])^+\Gamma_{\bm{S}}(j,\nu) \right] \geq  
 \mathbb{E}_\infty \left[ \sum\limits_{j=0}^{\tau^{(N)} -1}\sum\limits_{\nu=0}^{j}(1-W_{\bm \alpha}[\nu])^+\mathcal{L}_{\bm \alpha}(j,\nu)  \right],
\end{align}
which in turn together with \eqref{eq:eq5b} implies that
\begin{align}
\label{eq:eq15}
\mathrm{WADD}(\tau^{(N)}) &\geq 
 \nonumber\frac{ \mathbb{E}_\infty \left[ \sum\limits_{j=0}^{\tau^{(N)}-1} \sum\limits_{\nu=0}^{j}(1-W_{\bm \alpha}[\nu])^+ \mathcal{L}_{\bm \alpha}(j,\nu)  \right]}{\mathbb{E}_\infty  \left[\sum\limits_{\nu=0}^{\tau^{(N)}-1}(1-W_{\bm \alpha}[\nu])^+\right]}
\\&=
\frac{  \mathbb{E}_\infty \left[ \sum\limits_{j=0}^{\tau^{(N)}-1} \left(\sum\limits_{\nu=0}^{j-1}(1-W_{\bm \alpha}[\nu])^+ \mathcal{L}_{\bm \alpha}(j,\nu)+ (1-W_{\bm \alpha}[j])^+  \right) \right]}{\mathbb{E}_\infty  \left[\sum\limits_{\nu=0}^{\tau^{(N)}-1}(1-W_{\bm \alpha}[\nu])^+\right]}.
\end{align}
From Lemma 1 of \cite{moustakides:1986} we have that 
\begin{align}
\sum\limits_{\nu=0}^{j-1}(1-W_{\bm \alpha}[\nu])^+ \mathcal{L}_{\bm \alpha}(j,\nu) = W_{\bm \alpha}[j]
 \end{align}
 which together with \eqref{eq:eq15} implies that
\begin{align}
\label{eq:eq16}
&\mathrm{WADD}(\tau^{(N)})\geq  \frac{   \mathbb{E}_\infty \left[ \sum\limits_{j=0}^{\tau^{(N)}-1} \left(W_{\bm \alpha}[j]+ (1-W_{\bm \alpha}[j])^+  \right) \right]}{\mathbb{E}_\infty  \left[\sum\limits_{\nu=0}^{\tau^{(N)}-1}(1-W_{\bm \alpha}[\nu])^+\right]}
=
\frac{  \mathbb{E}_\infty \left[ \sum\limits_{j=0}^{\tau^{(N)}-1} \max\{ W_{\bm \alpha}[j],1\} \right]}{\mathbb{E}_\infty  \left[\sum\limits_{\nu=0}^{\tau^{(N)}-1}(1-W_{\bm \alpha}[\nu])^+\right]}.
\end{align}
Consider $b$ chosen such that $\mathbb{E}_\infty[\tau_W(\bm{\alpha},b)] = \gamma$. Let $b' \geq b$ such that $b ' >0$. Then, from Lemma \ref{thm:lemma_trunc} and eq. \eqref{eq:eq16} we have that
\begin{align}
\label{eq:eq17}
\mathrm{WADD}(\tau) &\geq \mathrm{WADD}(\tau^{(N)}) 
\geq
\frac{  \mathbb{E}_\infty \left[ \sum\limits_{j=0}^{\tau^{(N)}-1} \max\{ W_{\bm \alpha}[j],1\} \right]}{\mathbb{E}_\infty  \left[\sum\limits_{\nu=0}^{\tau^{(N)}-1}(1-W_{\bm \alpha}[\nu])^+\right]}
\nonumber \\&\geq 
\frac{  \mathbb{E}_\infty \left[ \sum\limits_{j=0}^{\tau^{(N)}-1}\min\{ \max\{ W_{\bm \alpha}[j],1\} ,e^{b'}\}\right]}{\mathbb{E}_\infty  \left[\sum\limits_{\nu=0}^{\tau^{(N)}-1}(1-W_{\bm \alpha}[\nu])^+\right]}.
\end{align}
Note that
\begin{align}
\bigg| \min \left\{\max\{W_{\bm \alpha}[j-1],1\},e^{b'}\right\} \bigg| \leq e^{b'}
\end{align}
and that since $W_{\bm \alpha}[j] \geq 0$
\begin{align}
|(1-W_{\bm \alpha}[j])^+| \leq 1.
\end{align}
Furthermore, since $\mathbb{E}_\infty[\tau] < \infty$ by assumption, by using Lemma \ref{converge} after taking the limit on both sides of \eqref{eq:eq17} we have that 
\begin{align}
\label{eq:eq18}
\mathrm{WADD}(\tau) 
&\geq 
\frac{  \mathbb{E}_\infty \left[ \sum\limits_{j=0}^{\tau-1}\min\{ \max\{ W_{\bm \alpha}[j],1\} ,e^{b'}\}\right]}{\mathbb{E}_\infty  \left[\sum\limits_{\nu=0}^{\tau-1}(1-W_{\bm \alpha}[\nu])^+\right]}
\end{align}
Since \eqref{eq:eq18} holds for arbitrary $\tau$ adapted to  $\mathscr{F}$, we have that for any $\gamma > 1$
\begin{align}
\label{eq:eq19}
\inf\limits_{\tau \,\in \,C_\gamma} \mathrm{WADD}
 &\geq 
 \inf\limits_{\tau \,\in \,C_\gamma} \frac{  \mathbb{E}_\infty \left[ \sum\limits_{j=0}^{\tau-1}\min\{ \max\{ W_{\bm \alpha}[j],1\} ,e^{b'}\}\right]}{\mathbb{E}_\infty  \left[\sum\limits_{\nu=0}^{\tau-1}(1-W_{\bm \alpha}[\nu])^+\right]}
  \nonumber \\& \geq 
\frac{ \inf\limits_{\tau \,\in \,C_\gamma}  \mathbb{E}_\infty \left[ \sum\limits_{j=0}^{\tau-1}\min\{ \max\{ W_{\bm \alpha}[j],1\} ,e^{b'}\}\right]}{\sup\limits_{\tau \, \in \,C_\gamma}  \mathbb{E}_\infty  \left[\sum\limits_{\nu=0}^{\tau-1}(1-W_{\bm \alpha}[\nu])^+\right]}.
\end{align}
Note that the function $\phi(x) = (1-x)^+$ in continuous and non-increasing with $\phi(0)=1$. As a result, from Theorem 1 of \cite{moustakides:1986} we have that 
\begin{align}
\label{eq:opt_1}
\mathbb{E}_\infty  \left[\sum\limits_{\nu=0}^{\tau_W(\bm{\alpha}, b)-1}(1-W_{\bm \alpha}[\nu])^+\right]
 =
  \sup\limits_{\tau \,\in\, C_\gamma}  \mathbb{E}_\infty  \left[\sum\limits_{\nu=0}^{\tau-1}(1-W_{\bm \alpha}[\nu])^+\right]
\end{align}
Furthermore, note that the function $\psi(x) = - \min \left\{\max\{x,1\},e^{b'}\right\}$ is continuous and non-increasing in $x$ with $\psi(0) = - \min \left\{1,\nu'\right\}$. As a result, from Theorem 1 of \cite{moustakides:1986} we also have that
\begin{align}
\label{eq:opt_2}
 &\inf\limits_{\tau \,\in \,C_\gamma}  \mathbb{E}_\infty \left[ \sum\limits_{j=0}^{\tau-1}\min\{ \max\{ W_{\bm \alpha}[j],1\} ,e^{b'}\}\right]=\nonumber
  -\sup\limits_{\tau\, \in \,C_\gamma}  \mathbb{E}_\infty \left[- \sum\limits_{j=0}^{\tau-1}\min\{ \max\{ W_{\bm \alpha}[j],1\} ,e^{b'}\}\right]
  \\& =
 - \mathbb{E}_\infty \left[- \sum\limits_{j=0}^{\tau_W(\bm{\alpha},b)-1}\min\{ \max\{ W_{\bm \alpha}[j],1\} ,e^{b'}\}\right]
  =
  \mathbb{E}_\infty \left[\sum\limits_{j=0}^{\tau_W(\bm{\alpha},b)-1}\min\{ \max\{ W_{\bm \alpha}[j],1\} ,e^{b'}\}\right].
 \end{align}
 Then, from \eqref{eq:eq19} - \eqref{eq:opt_2} we have that
\begin{align}
\label{eq:eq20}
\inf\limits_{\tau \in C_\gamma}\mathrm{WADD}(\tau) 
&\geq  
\frac{\mathbb{E}_\infty \left[\sum\limits_{j=0}^{\tau_W(\bm{\alpha},b)-1}\min\{ \max\{ W_{\bm \alpha}[j],1\} ,e^{b'}\}\right]} {\mathbb{E}_\infty  \left[\sum\limits_{\nu=0}^{\tau_W(\bm{\alpha},b)-1}(1-W_{\bm \alpha}[\nu])^+\right]} 
\nonumber \\&\stackrel{(f)}{=} 
\frac{\mathbb{E}_\infty \left[\sum\limits_{j=0}^{\tau_W(\bm{\alpha},b)-1} \max\{ W_{\bm \alpha}[j],1\} \right]} {\mathbb{E}_\infty  \left[\sum\limits_{\nu=0}^{\tau_W(\bm{\alpha},b)-1}(1-W_{\bm \alpha}[\nu])^+\right]} ,
\end{align}
where $(f)$ is implied since $W_{\bm \alpha}[j] < e^{b} \leq e^{b'}$ for $0 \leq j <\tau_W(\bm{\alpha},b)$ and since $b' >0$. 
Furthermore, note that from the optimality of the CUSUM test for the classic QCD problem \cite{moustakides:1986} we have that
\begin{align}
\label{eq:eq20b}
\frac{\mathbb{E}_\infty \left[\sum\limits_{j=0}^{\tau_W(\bm{\alpha},b)-1} \max\{ W_{\bm \alpha}[j],1\} \right]} {\mathbb{E}_\infty  \left[\sum\limits_{\nu=0}^{\tau_W(\bm{\alpha},b)-1}(1-W_{\bm \alpha}[\nu])^+\right]} 
= 
\overline{\mathrm{WADD}}(\tau_W(\bm{\alpha},b)).
\end{align}
As a result, from \eqref{eq:eq20} and \eqref{eq:eq20b} and since 
\begin{align}
\mathrm{WADD}(\tau_W(\bm{\alpha},b)) \geq \inf\limits_{\tau \in C_\gamma}\mathrm{WADD}(\tau)
\end{align}
the theorem is established.
\end{proof}
\begin{proof}[Proof of Lemma 1]
Fix $\bm{\alpha} \in \mathcal{A}$, $b>0$ and $N>0$. For purposes of presentation of this proof, we denote the stopping $\tau_W(\bm{\lambda}_U,b)$ with uniform weights and threshold $b$ by simply $\tau_W$ and $W_{\bm{\lambda}_U}[k]$, $\mathcal{L}_{\bm{\alpha}}(\cdot,\cdot)$ by $W[k]$ and $\mathcal{L}(\cdot,\cdot)$ respectively. Define the truncated stopping time  $\tau_{W}^{(N)} = \min\{\tau_W, N\}$. Note that by employing a change of measure similar to the one in \eqref{eq:eq1} we have that for any $\nu \geq 0$ and any $\bm{S}$
\begin{align}
\label{eq:eq22}
 V_\nu & \triangleq \mathbb{E}_\nu^{\bm{S}} \left[\tau_{W}^{(N)}-\nu \Big|\tau_{W}^{(N)}> \nu, \mathscr{F}_{\nu} \right] 
 =
 \mathbb{E}_\nu^{\bm{S}} \left[\sum_{j=\nu}^{\infty}\mathbbm{1}_{\left\{\tau_{W}^{(N)} > j\right\}} \bigg|\tau_{W}^{(N)}> \nu, \mathscr{F}_{\nu}\right]
 \nonumber \\&
 =
 \mathbb{E}_\nu^{\bm{S}} \left[\sum_{j=\nu}^{N-1}\mathbbm{1}_{\left\{\tau_{W}^{(N)}> j\right\}}  \bigg|\tau_{W}^{(N)}>\nu, \mathscr{F}_{\nu} \right]\nonumber
 =\mathbb{E}_\infty \left[\sum_{j=\nu}^{N-1}\Gamma_{\bm{S}}(j,\nu)\mathbbm{1}_{\left\{\tau_{W}^{(N)}>j\right\}} \bigg|\tau_{W}^{(N)}> \nu, \mathscr{F}_{\nu} \right]
 \\&=
\nonumber  1+ \mathbb{E}_\infty \left[\sum_{j=\nu+1}^{N-1}\Gamma_{\bm{S}}(j,\nu)\mathbbm{1}_{\left\{\tau_{W}^{(N)}> j\right\}} \bigg|\tau_{W}^{(N)}>\nu, \mathscr{F}_{\nu} \right]
   \\& \stackrel{(a)}{=}
 1+ \mathbb{E}_\infty \left[\sum_{j=\nu+1}^{N-1}\Gamma_{\bm{S}}(j,\nu)\left(\prod_{i=\nu+1}^{j}\mathbbm{1}_{\{W[i]<e^b\}}\right) \bigg|\tau_{W}^{(N)}> \nu, \mathscr{F}_{\nu} \right] ,
\end{align}
where $(a)$ follows since for $\nu< j < N$ we have that conditioned on $\left\{\tau_{W}^{(N)} > \nu\right\}$ 
\begin{align}
\left\{\tau_{W}^{(N)}> j \right\}=\bigcap_{i=\nu+1}^j\{W[i]<e^b\}. 
\end{align}
To proceed, we establish that for any $0 \leq \nu \leq N-1$ 
\begin{align}
\label{eq:eq23}
V_\nu = 1+\mathbb{E}_\infty\left[ \Gamma_{\bm{S}}(\nu+1,\nu)\mathbbm{1}_{\{W[\nu+1]<e^b\}}V_{\nu+1} \bigg| \tau_{W}^{(N)}>\nu, \mathscr{F}_{\nu} \right],
\end{align}
with $V_{\nu}=1$ for all $\nu\geq N-1$. First of all, from the definition of $V_\nu$ we have that
\begin{align}
V_{N-1}=\mathbb{E}_{N-1}^{\bm{S}} \left[\tau_{W}^{(N)}-N +1\Big|\tau_{W}^{(N)} > N-1, \mathscr{F}_{N-1} \right] =
\mathbb{E}_{N-1}^{\bm{S}} \left[N-N+1 \Big|\tau_{W}^{(N)}> N, \mathscr{F}_{N-1} \right] 
=1.
\end{align}
In addition, for $\nu \geq N$ the event $\{\tau_{W}^{(N)}> \nu\}$ cannot occur, hence, by convention we have that in this case $V_\nu =1$.
Furthermore, note that $\Gamma_{\bm{S}}(\nu+1,\nu)\mathbbm{1}_{\{W[\nu+1]<e^b\}}$ is present in all terms of the summation in \eqref{eq:eq22}, hence
{\small
\begin{align}
\label{eq:eq25}
V_\nu 
&=  
1+ \nonumber \mathbb{E}_\infty \left[\Gamma_{\bm{S}}(\nu+1,\nu)\mathbbm{1}_{\{W[\nu+1]<e^b\}}\sum_{j=\nu+1}^{N-1}\Gamma_{\bm{S}}(j,\nu+1)\left(\prod_{i=\nu+2}^{j}\mathbbm{1}_{\{W[i]<e^b\}}\right) \bigg|\tau_{W}^{(N)}> \nu, \mathscr{F}_{\nu} \right]
\\&
=
1+ \nonumber \mathbb{E}_\infty \bigg[\Gamma_{\bm{S}}(\nu+1,\nu)\mathbbm{1}_{\{W[\nu+1]<e^b\}}\bigg(1
+
\sum_{j=\nu+2}^{N-1}\Gamma_{\bm{S}}(j,\nu+1)\left(\prod_{i=\nu+2}^{j}\mathbbm{1}_{\{W[i]<e^b\}}\right) \bigg)
\nonumber\\&
\bigg|\tau_{W}^{(N)}> \nu, \mathscr{F}_{\nu} \bigg] \nonumber
\\&  
\stackrel{(b)}{=}
1+ \nonumber\mathbb{E}_\infty\bigg[ \mathbb{E}_\infty \bigg[\Gamma_{\bm{S}}(\nu+1,\nu)\mathbbm{1}_{\{W[\nu+1]<e^b\}}\bigg(1+\sum_{j=\nu+2}^{N-1}\Gamma_{\bm{S}}(j,\nu+1)\bigg(\prod_{i=\nu+2}^{j}\mathbbm{1}_{\{W[i]<e^b\}}\bigg) \bigg)
\nonumber\\&
\bigg|\tau_{W}^{(N)}> \nu+1, \mathscr{F}_{\nu+1}\bigg] \bigg| \tau_{W}^{(N)}>\nu, \mathscr{F}_{\nu} \bigg] \nonumber
\\&\stackrel{(c)}{=} 
1+ \nonumber\mathbb{E}_\infty\bigg[ \Gamma_{\bm{S}}(\nu+1,\nu)\mathbbm{1}_{\{W[\nu+1]<e^b\}}\bigg(1+\mathbb{E}_\infty \bigg[\sum_{j=\nu+2}^{N-1}\Gamma_{\bm{S}}(j,\nu+1)\bigg(\prod_{i=\nu+2}^{j}\mathbbm{1}_{\{W[i]<e^b\}}\bigg)
\nonumber\\&
 \bigg|\tau_{W}^{(N)}>\nu+1, \mathscr{F}_{\nu+1}\bigg] \bigg)\bigg| \tau_{W}^{(N)}> \nu, \mathscr{F}_{\nu} \bigg]\nonumber 
\\& \stackrel{(d)}{=} 
1+\mathbb{E}_\infty\left[ \Gamma_{\bm{S}}(\nu+1,\nu)\mathbbm{1}_{\{W[\nu+1]<b\}}V_{\nu+1} \bigg| \tau_{W}^{(N)} > \nu, \mathscr{F}_{\nu} \right],
\end{align}
}where $(b)$ follows from the tower property of expectations, $(c)$ follows since $ \Gamma_{\bm{S}}(\nu+1,\nu)\mathbbm{1}_{\{W[\nu+1]<b\}}$ is $\mathscr{F}_{\nu+1}$-measurable, and hence can go out of the conditional expectation, and $(d)$ follows from \eqref{eq:eq22}. 

We will now establish that $V_\nu$ is independent of $\bm{S}$ for all $ \nu \geq 0$ and that it is a function of $\mathscr{F}_{\nu}$ only through $W[\nu]$. First of all, note that for $\nu\geq N -1$, $V_\nu=1$ so we only have to investigate the case that $\nu \leq N-2$. For $\nu \leq N-2$ since $\tau_{W}^{(N)}$ is truncated by $N$ and since  $\bm{X}[1],\dots,\bm{X}[\nu]$ are independent from $\bm{S}$ we have to show that $V_\nu$ is independent of $\bm{S}[\nu+1,N]$ and that $V_\nu$ is a function of $\mathscr{F}_\nu$ only through $W[\nu]$. For $2 \leq k \leq N-2$, assume that the statement holds for $V_{N-k}(W[N-k])$. From \eqref{eq:eq23} we have that 
{\small
\begin{align}
\label{eq:eq26}
&V_{N-(k+1)} 
= 
1+\mathbb{E}_\infty\left[ \Gamma_{\bm{S}}(N-k,N-k-1)\mathbbm{1}_{\{W[N-k]<e^b\}}V_{N-k}(W[N-k]) \bigg| \tau_{W}^{(N)} > N-k-1, \mathscr{F}_{N-k-1} \right]
\nonumber\\& 
\stackrel{(e)}{=}  
1+\mathbb{E}_\infty\bigg[\Gamma_{\bm{S}}(N-k,N-k-1)\mathbbm{1}_{\{\max\{W[N-k-1],1\}\mathcal{L}\left(N-k,N-k-1\right)<e^b\}} \nonumber \\&V_{N-k}(\max\{W[N-k-1],1\}\mathcal{L}(N-k,N-1-k))  \Big|\tau_{W}^{(N)}\geq N-k-1, \mathscr{F}_{N-k-1} \bigg] 
\nonumber\\& 
\stackrel{(f)}{=} 
1+\mathbb{E}_\infty\bigg[\left(\prod\limits_{\ell\,\in\, \bm{S}[N-k]}\frac{f(X_\ell[N-k])}{g(X_\ell[N-k])}\right)  \mathbbm{1}_{\bigg\{\max\{W[N-k-1],1\}\left(\prod\limits_{\ell\,\in\, \bm{S}[N-k]}\frac{f(X_\ell[N-k])}{g(X_\ell[N-k])}\right) <e^b\bigg\}} \nonumber \\&V_{N-k}\left(\max\{W[N-k-1],1\}\left(\prod\limits_{\ell\,\in\, \bm{S}[N-k]}\frac{f(X_\ell[N-k])}{g(X_\ell[N-k])}\right)\right)  \Bigg|\tau_{W}^{(N)}> N-k-1, \mathscr{F}_{N-k-1} \bigg] 
 ,
\end{align}
}
where (e) follows from eq. \eqref{eq:test_rec} and (f) follows from \eqref{eq:like_ratio}.  Note that under $\mathbb{P}_\infty(\cdot)$, the distribution of the likelihood ratio in \eqref{eq:eq26} is independent of $\bm{S}[N-k]$. As a result, we have that for all $\bm{E} \in \mathcal{E}$
\begin{align}
\label{eq:eqimp}
&V_{N-(k+1)} 
= 
1+\mathbb{E}_\infty\bigg[\left(\prod\limits_{\ell\,\in\, \bm{E}}\frac{f(X_\ell[N-k])}{g(X_\ell[N-k])}\right)  \mathbbm{1}_{\bigg\{\max\{W[N-k-1],1\}\left(\prod\limits_{\ell\,\in\, \bm{E}}\frac{f(X_\ell[N-k])}{g(X_\ell[N-k])}\right) <e^b\bigg\}} \nonumber \\&V_{N-k}\left(\max\{W[N-k-1],1\}\left(\prod\limits_{\ell\,\in\, \bm{E}}\frac{f(X_\ell[N-k])}{g(X_\ell[N-k])}\right)\right)  \Bigg|\tau_{W}^{(N)}> N-k-1, \mathscr{F}_{N-k-1} \bigg] 
 .
\end{align}
From \eqref{eq:eqimp} we can then easily see that $V_{N-(k+1)}$ is independent of $\bm{S}$. Furthermore, since the likelihood ratio in the last line is independent of $\mathscr{F}_{N-k-1}$ we have that $V_{N-(k+1)}$ is a function of $\mathscr{F}_{N-k-1}$ only through $W_{N-k-1}$. As a result, by induction we have that for all $\nu \geq 0$, $V_\nu$ is independent of $\bm{S}$ and depends on $\mathscr{F}_{\nu}$ only through $W[\nu]$. 

Following, note that for $\nu \geq 0$, from the independence of $V_{\nu}$ from $\bm{S}$ and eq. \eqref{eq:eq23} we have that for all $\bm{E} \in \mathcal{E}$
\begin{align}
\label{eq:eksiswsh2}
V_\nu = 1+\mathbb{E}_\infty\left[ \left(\prod\limits_{\ell\,\in\, \bm{E}}\frac{f(X_\ell[\nu+1])}{g(X_\ell[\nu+1])}\right) \mathbbm{1}_{\{W[\nu+1]<e^b\}}V_{\nu+1} \bigg| \tau_{W}^{(N)}>\nu, \mathscr{F}_{\nu} \right].
\end{align}
As a result, by averaging over $\bm{E}$ with respect to $\bm{\alpha}$ we have that
\begin{align}
\label{eq:eksiswsh3}
V_{\nu} \nonumber
\nonumber&= 1+ \sum\limits_{\bm{E}\,\in\,\mathcal{E}}\alpha_{\bm{E}}\mathbb{E}_\infty\left[ \left(\prod\limits_{\ell\,\in\, \bm{E}}\frac{f(X_\ell[\nu+1])}{g(X_\ell[\nu+1])}\right) \mathbbm{1}_{\{W[\nu+1]<e^b\}}V_{\nu+1} \bigg| \tau_{W}^{(N)} > \nu, \mathscr{F}_{\nu} \right] 
\nonumber
\\&= 
1+\mathbb{E}_\infty\left[ \left( \sum\limits_{\bm{E}\,\in\,\mathcal{E}}\alpha_{\bm{E}}\prod\limits_{\ell\,\in\, \bm{E}}\frac{f(X_\ell[\nu+1])}{g(X_\ell[\nu+1])}\right) \mathbbm{1}_{\{W[\nu+1]<e^b\}}V_{\nu+1} \bigg| \tau_{W}^{(N)}> \nu, \mathscr{F}_{\nu} \right] 
\nonumber
\\& 
=
1+\mathbb{E}_\infty\left[ \mathcal{L}(\nu+1,\nu) \mathbbm{1}_{\{W[\nu+1]<e^b\}}V_{\nu+1} \bigg|  \tau_{W}^{(N)}> \nu, \mathscr{F}_{\nu}\right] .
\end{align}
By unfolding the recursion in \eqref{eq:eksiswsh3}, it can be easily seen that for any $\nu \geq 0$ and any  $\bm{S}$
\begin{align}
\label{eq:eqabc}
\mathbb{E}_\nu^{\bm{S}} \left[\tau_{W}^{(N)}-\nu|\tau_{W}^{(N)}>\nu , \mathscr{F}_{\nu} \right] 
= 
\overline{\mathbb{E}}_\nu^{\bm{\alpha}} \left[\tau_{W}^{(N)}-\nu|\tau_{W}^{(N)}> \nu , \mathscr{F}_{\nu} \right].
\end{align}
From the Monotone Convergence Theorem, since $\tau_{W}^{(N)}-\nu$ and $\mathbbm{1}_{\left\{\tau_{W}^{(N)}-\nu\right\}}$ are non-decreasing with $N$, we have that for all $\bm{S}$
\begin{align}
\label{eq:eq28}
 &\lim_{N \rightarrow \infty}\mathbb{E}_\nu^{\bm{S}} \left[\tau_{W}^{(N)}-\nu \Big|\tau_{W}^{(N)} >\nu , \mathscr{F}_{\nu} \right]
 = 
 \lim_{N \rightarrow \infty}\frac{\mathbb{E}_\nu^{\bm{S}} \left[(\tau_{W}^{(N)}-\nu)\mathbbm{1}_{\{\tau_{W}^{(N)}> \nu\}} \Big| \mathscr{F}_{\nu} \right]}{\mathbb{E}_\nu^{\bm{S}} \left[\mathbbm{1}_{\left\{\tau_{W}^{(N)}> \nu\right\}}\Big| \mathscr{F}_{\nu} \right] }
 \nonumber\\&= 
\frac{\mathbb{E}_\nu^{\bm{S}} \left[\lim_{N \rightarrow \infty}(\tau_{W}^{(N)}-\nu)\mathbbm{1}_{\{\tau_{W}^{(N)}>\nu\}} \Big| \mathscr{F}_{\nu} \right]}{\mathbb{E}_\nu^{\bm{S}} \left[\lim_{N \rightarrow \infty}\mathbbm{1}_{\{\tau_{W}^{(N)} > \nu\}}\Big| \mathscr{F}_{\nu} \right] }  \nonumber
 =
\frac{\mathbb{E}_\nu^{\bm{S}} \left[(\tau_{W}-\nu)\mathbbm{1}_{\{\tau_{W}> \nu\}} |\mathscr{F}_{\nu} \right]}{\mathbb{E}_\nu^{\bm{S}} \left[\mathbbm{1}_{\{\tau_{W}> \nu\}}| \mathscr{F}_{\nu} \right] } 
\\&=
\mathbb{E}_\nu^{\bm{S}} \left[\tau_{W}-\nu|\tau_{W}> \nu , \mathscr{F}_{\nu} \right].
\end{align}
Similarly, it can be shown that
\begin{align}
\label{eq:eq299}
  \lim_{N \rightarrow \infty}\overline{\mathbb{E}}_\nu^{\bm{\alpha}} \left[\tau_{W}^{(N)}-\nu \Big|\tau_{W}^{(N)} > \nu , \mathscr{F}_{\nu} \right]
  =
\overline{\mathbb{E}}_\nu^{\bm{\alpha}}  \left[\tau_{W}-\nu|\tau_{W} > \nu , \mathscr{F}_{\nu} \right].
\end{align}
As a result, by taking the limit on both sides of \eqref{eq:eqabc} and using eqs. \eqref{eq:eq28} and \eqref{eq:eq299} we have that for all $\nu \geq 0$, $\bm{S}$
\begin{align}
\overline{\mathbb{E}}_\nu \left[\tau_{W}-\nu|\tau_{W}> \nu , \mathscr{F}_{\nu} \right]
=  
\mathbb{E}_\nu^{\bm{S}} \left[\tau_{W}-\nu|\tau_{W}> \nu , \mathscr{F}_{\nu} \right]
\end{align}
which in turn implies 
\begin{align}
\mathrm{WADD}(\tau_W) = \overline{\mathrm{WADD}}_{\bm{\alpha}}(\tau_W) .
\end{align}
\end{proof}
\begin{proof}[Proof of Lemma 2]
Define $\bm{\beta} = \left[\beta_{\bm{E}_1},\dots,\beta_{\bm{E}_{|\mathcal{E}|-1}}\right]^\top$ where $\alpha_{\bm{E}_j} \triangleq \beta_{\bm{E}_j}$ for $j \in  \left[|\mathcal{E}|-1\right]$. The constrained optimization of $I_{\bm{\alpha}}$ can then be equivalently replaced by
\begin{equation}
\label{eq:minim_1}
\begin{aligned}
& \underset{\bm{\beta}}{\text{inf}}
& &q(\bm{\beta})\\
& \text{s.t.}
& & \beta_{\bm{E}_j} \geq 0,\,\, \forall\,\,  j \in \left[ |\mathcal{E}| -1 \right]\\
&
&&\sum_{j=1}^{|\mathcal{E}| -1} \beta_{\bm{E}_j} \leq 1,
\end{aligned}
\end{equation}
where
\begin{align}
 &q(\bm{\beta})\hspace{-0.05cm}\triangleq \hspace{-0.1cm} \int\limits_{\mathbb{R}^L} \left( \left(1-\|\bm{\beta}\|_1\right) p_{\bm{E}_{|\mathcal{E}|} }(\bm{x}) + \sum\limits_{j =1 }^{|\mathcal{E}| -1}\beta_{\bm{E}_j} p_{\bm{E}_j}(\bm{x})\right)
\nonumber 
\\&
 \log\hspace{-0.1cm} \left(\frac{\left( \left(1-\|\bm{\beta}\|_1\right) p_{\bm{E}_{|\mathcal{E}|} }(\bm{x}) + \sum\limits\limits_{j =1 }^{|\mathcal{E}|-1}\beta_{\bm{E}_j} p_{\bm{E}_j}(\bm{x})\right)}{g(\bm{x})}\right) .
\end{align}
Denote by $\bm{\beta}^*$ the solution to \eqref{eq:minim_1}. Then, the derivative at $\bm{\beta}^*$ is given by
\begin{align}
\label{eq:eksiswsoyla}
\frac{\partial q(\bm{\beta})}{\partial \beta_{\bm{E}_i}}\bigg|_{\bm{\beta}^*}& = \mathbb{E}_{p_{\bm{E}_i}}\left[\log\left(\frac{\overline{p}_{\bm{\alpha}^*}(\bm{X})}{g(\bm{X})} \right)\right] - \mathbb{E}_{p_{\bm{E}_{|\mathcal{E}|}}}\left[\log\left(\frac{\overline{p}_{\bm{\alpha}^*}(\bm{X})}{g(\bm{X})} \right)\right] .
\end{align}
WLOG we have that that either $\bm{\beta}^* = [\beta^*_{\bm{E}_1},\ldots, \beta^*_{\bm{E}_\eta},\dots,0]^\top$ with $\eta \in [|\mathcal{E}| - 1]$ and $\beta^*_{\bm{E}_j}>0$ for all $j \in [\eta]$ (boundary or interior point), or $\bm{\beta}^* = [0,\ldots,0]^\top$ (corner point). 

Assume that $\bm{\beta}^*$ is a corner point. Denote by $D(f \| g)$ denote the KL-divergence between two pdfs $f(\cdot)$ and $g(\cdot)$. In this case we have that for all $i \in [|\mathcal{E}|-1]$
\begin{align}
&\frac{\partial q(\bm{\beta})}{\partial \beta_{\bm{E}_i}}\bigg|_{\bm{\beta}^*} 
=\sum\limits_{\ell\,\in \,\bm{E}_{|\mathcal{E}|}} \left( D(f_\ell \| g_\ell) \mathbbm{1}_{\ell \,\in \, \bm{E}_i} -   D(g_\ell \| f_\ell) \mathbbm{1}_{\ell \, \notin \, \bm{E}_i} \right)
- \sum\limits_{\ell \,\in \, \bm{E}_{|\mathcal{E}|}}  D(f_i \| g_i)<0,
\end{align}
which is a contradiction since 
\begin{align}
\frac{\partial q(\bm{\beta})}{\partial \beta_{\bm{E}_i}}\bigg|_{\bm{\beta}^*}\geq 0
\end{align} 
must hold for all $i \in [|\mathcal{E}|-1]$ due to the fact that $\bm{\beta}^*$ is a minimum. 

As a result, $\bm{\beta}^*$ is not a corner point. In this case, for all $i \in [\eta]$ we have that
\begin{align}
\frac{\partial q(\bm{\beta})}{\partial \beta_{\bm{E}_i}}\bigg|_{\bm{\beta}^*} =0,
\end{align}
which implies that for all $i \in [\eta]$
\begin{align}
\label{eq:eksiswsoula_2}
 \mathbb{E}_{p_{\bm{E}_i}}\left[\log\left(\frac{\overline{p}_{\bm{\alpha}^*}(\bm{X})}{g(\bm{X})} \right)\right]= \mathbb{E}_{p_{\bm{E}_{|\mathcal{E}|}}}\left[\log\left(\frac{\overline{p}_{\bm{\alpha}^*}(\bm{X})}{g(\bm{X})} \right)\right]  \triangleq J.
\end{align}
Furthermore, we have that since $\alpha_{\bm{E}_j}^* =0 $ for $ \eta < j <|\mathcal{E}|$
\begin{align}
\label{eq:eksiswsoula_3}
J &= \left(\sum\limits_{j = 1}^{\eta}  \beta_{\bm{E}_j}^* + \left(1 - \sum\limits_{j = 1}^{\eta}\beta_{\bm{E}_j}^*\right)  \right)J \nonumber
=\left(\sum\limits_{j = 1}^{\eta} \alpha_{\bm{E}_j}^*  +\alpha_{\bm{E}_{|\mathcal{E}|}}^*\right)J
 \\&
= \sum\limits_{j=1}^{|\mathcal{E}|} \alpha_{\bm{E}_j}^* \mathbb{E}_{p_{\bm{E}_{j}}}\left[\log\left(\frac{\overline{p}_{\bm{\alpha}^*}(\bm{X})}{g(\bm{X})} \right)\right]=
 \mathbb{E}_{\overline{p}_{\bm{\alpha}^*}}\left[ \log\left(\frac{\overline{p}_{\bm{\alpha}^*}(\bm{X})}{g(\bm{X})} \right)\right] = I_{\bm{\alpha}^*}>0.
\end{align}
In addition, we have that for $\eta < i < |\mathcal{E}|$
\begin{align}
\frac{\partial q(\bm{\beta})}{\partial \beta_{\bm{E}_i}}\bigg|_{\bm{\beta}^*} > 0.
\end{align}
This implies that for all $i \in [\eta] \cup \{|\mathcal{E}|\}$ and $\eta < j < |\mathcal{E}|$
\begin{align}
 \mathbb{E}_{p_{\bm{E}_j}}\left[\log\left(\frac{\overline{p}_{\bm{\alpha}^*}(\bm{X})}{g(\bm{X})} \right)\right]>  \mathbb{E}_{p_{\bm{E}_i}}\left[\log\left(\frac{\overline{p}_{\bm{\alpha}^*}(\bm{X})}{g(\bm{X})} \right)\right]  = I_{\bm{\alpha}^*}.
\end{align}
ii) For the case of $m=1$, WLOG assume that for all $1 \leq j \leq |\mathcal{E}| = L$, we have that $E_j = j$. For $ \eta < i < L$, we then have that 
\begin{align}
\label{eq:opt_eq_3}
& \mathbb{E}_{p_{E_i}}\left[\log\left(\frac{\overline{p}_{\bm{\alpha}^*}(\bm{X})}{g(\bm{X})} \right)\right] =  \mathbb{E}_{p_{i}}\left[\log\left(\frac{\overline{p}_{\bm{\alpha}^*}(\bm{X})}{g(\bm{X})} \right)\right] = 
 \mathbb{E}_{p_{i}}\left[\log\left(\sum\limits_{j=1}^{\eta}\alpha_j^*\frac{f_{j}(X_j)}{g_j(X_j)} +\alpha_L^*\frac{f_{L}(X_L)}{g_L(X_L)} \right)  \right]   
 \nonumber \\&
 = \mathbb{E}_{g}\left[\log\left(\sum\limits_{j=1}^{\eta}\alpha_j^*\frac{f_{j}(X_j)}{g_j(X_j)} +\alpha_L^*\frac{f_{L}(X_L)}{g_L(X_L)} \right)  \right]   =  \mathbb{E}_{g}\left[\log\left(\frac{\overline{p}_{\bm{\alpha}^*}(\bm{X})}{g(\bm{X})} \right)\right]<0.
\end{align}
We then have that from eqs. \eqref{eq:eksiswsoyla}, \eqref{eq:eksiswsoula_2}, \eqref{eq:eksiswsoula_3} and \eqref{eq:opt_eq_3}
\begin{align}
\label{eq:eqeqeqeqeq}
\frac{\partial q(\bm{\beta})}{\partial \beta_{E_i}}\bigg|_{\bm{\beta}^*} <0
\end{align} 
for all $ \eta < i < L$, which leads to a contradiction, since \eqref{eq:eqeqeqeqeq} cannot hold at the minimum.
\end{proof}
\begin{proof}[Proof of Theorem 4]
Our upper bound analysis is based on the proof technique in \cite{lai-ieeetit-1998}. Due to the structure of the test we have that for any $b>0$
\begin{align}
\mathrm{WADD}(\tau_W(\bm{\alpha}^*, b)) = \sup_{\bm S} \mathbb{E}_0^{\bm S}[\tau_W(\bm{\alpha}^*, b)].
\end{align}
Let $0< \epsilon<I_{\bm{\alpha}^*}$ and $n_b = \frac{b}{I_{\bm{\alpha}^*}-\epsilon}$. We then have that
\begin{align}
\sup\limits_{\bm S}\mathbb{E}_0^{\bm S}\left[\frac{\tau_W(\bm{\alpha}^*, b)}{n_b}\right]& \stackrel{(a)}{=} \sup\limits_{\bm S}\int\limits_{0}^\infty\mathbb{P}_0^{\bm S}\left(\frac{\tau_W(\bm{\alpha}^*, b)}{n_b} >x \right) dx \nonumber
  \stackrel{(b)}{\leq} \sup\limits_{\bm S}\sum_{\zeta = 0}^\infty \mathbb{P}_0^{\bm S} (\tau_W(\bm{\alpha}^*, b) > \zeta n_b)
\nonumber \\&=  1+  \sup\limits_{\bm S}\sum_{\zeta = 1}^\infty \mathbb{P}_0^{\bm S} (\tau_W(\bm{\alpha}^*, b) > \zeta n_b)
,
\label{eq:initial_eq}
\end{align}
where (a) follows from writing the expectation as an integral of the inverse cumulative density function for a positive random variable and (b) from the sum-integral inequality. 

Define the log-likelihood ratio at time $j$ corresponding to \eqref{eq:eq:random_stat_model} for $\bm{\alpha} = \bm{\alpha}^*$ by
\begin{align}
Z_{{\bm{\alpha}}^*}[j] \triangleq \log\frac{\overline{p}_{\bm{\alpha}^*}(\bm{X}[j])}{g(\bm{X}[j])}.
\end{align}
For any path $\bm S= \{\bm S[k]\}_{k=1}^\infty$, $\zeta \geq 1$, we then have that
\begin{align}
&\mathbb{P}_0^{\bm S} (\tau_W(\bm{\alpha}^*, b) > \zeta n_b) = \nonumber\mathbb{P}_0^{ \bm S} \left(\max\limits_{1\leq k \leq \zeta n_b}W_{{\bm \alpha}^*}[k] <e^b\right) 
\stackrel{(c)}{=} 
\mathbb{P}_0^{ \bm S} \left(\max\limits_{1\leq k \leq \zeta n_b}\max\limits_{1\leq i \leq k} \mathcal{L}_{{\bm{\alpha}}^*}(k,i-1) <e^b\right) 
\nonumber\\&
\stackrel{(d)}{=}  \mathbb{P}_0^{\bm S} \left(\max_{1 \leq k \leq \zeta n_b}   \max\limits_{1\leq i \leq k }\sum\limits_{j=i}^k Z_{\bm{\alpha}^*}[j]  < b  \right) \nonumber 
\stackrel{(e)}{\leq}  \mathbb{P}_0^{\bm S} \left( \max\limits_{1\leq i \leq r n_b} \sum\limits_{j=i}^{r n_b}Z_{\bm{\alpha}^*}[j] < b ,\, \forall \, r \in [\zeta]  \right) \nonumber
\\& \stackrel{(f)}{\leq}  \mathbb{P}_0^{\bm S} \left( \sum\limits_{j=(r-1)n_b +1}^{r n_b}Z_{\bm{\alpha}^*}[j]    < b ,\, \forall \, r \in [\zeta]  \right) 
\stackrel{(g)}{=} \mathbb{P}_0^{\bm S} \left( \frac{\sum\limits_{j=(r-1)n_b +1}^{r n_b} Z_{\bm{\alpha}^*}[j] }{n_b}   < I_{\bm{\alpha}^*} -\epsilon,\, \forall \, r \in [\zeta]  \right) 
 \nonumber 
 \\& \stackrel{(h)}{=}  \prod_{r=1}^\zeta  \mathbb{P}_0^{\bm S} \left( \frac{\sum\limits_{j=(r-1)n_b +1}^{r n_b}Z_{\bm{\alpha}^*}[j]}{n_b}  < I_{\bm{\alpha}^*} -\epsilon \right) ,
 \label{eq:initial_eq_eq}
\end{align} 
where (c) follows from the definition of the M-CUSUM statistic (eq. \eqref{eq:test_stat}), (d) follows by taking the exponent at both sides of the inequality, (e) and (f) by using the binning technique in \cite{lai-ieeetit-1998}, (g) by diving both sides by $n_b$ and (h) by the independence of the observations over time. Note that for $b>0$ we then have that from eqs. \eqref{eq:initial_eq} and \eqref{eq:initial_eq_eq}
\begin{align}
&\sup_{\bm S } \sum_{\zeta = 1}^\infty \mathbb{P}_0^{\bm S} (\tau_W(\bm{\alpha}^*, b) > \zeta n_b)\nonumber 
= \sup_{\bm S} \lim_{\xi \rightarrow \infty}\sum_{\zeta = 1}^{\xi}  \mathbb{P}_0^{\bm S} (\tau_W(\bm{\alpha}^*, b) > \zeta n_b)  \nonumber
 \\&\leq  \lim_{\xi \rightarrow \infty}\sup_{\bm S} \sum_{\zeta = 1}^{\xi}  \mathbb{P}_0^{\bm S} (\tau_W(\bm{\alpha}^*, b) > \zeta n_b) \nonumber 
\leq  \lim_{\xi \rightarrow \infty}\sum_{\zeta = 1}^{\xi}  \sup_{\bm S }\mathbb{P}_0^{\bm S} (\tau_W(\bm{\alpha}^*, b)> \zeta n_b) \nonumber
 \\& \leq \lim_{\xi \rightarrow \infty}\sum_{\zeta = 1}^\xi \sup_{\bm S } \left[\prod_{r=1}^\zeta  \mathbb{P}_0^{\bm S} \left( \frac{\sum\limits_{j=(r-1)n_b +1}^{r n_b}Z_{\bm{\alpha}^*}[j]  }{n_b}  < I_{\bm{\alpha}^*} -\epsilon \right) \right] \nonumber 
 \\& 
 \leq  \lim_{\xi \rightarrow \infty}\sum_{\zeta = 1}^\xi  \prod_{r=1}^\zeta \left[\sup_{\bm S }   \mathbb{P}_0^{\bm S} \left( \frac{\sum\limits_{j=(r-1)n_b +1}^{r n_b}Z_{\bm{\alpha}^*}[j] }{n_b}  < I_{\bm{\alpha}^*} -\epsilon \right) \right] \nonumber 
 \\& 
 =     \lim_{\xi \rightarrow \infty}\sum_{\zeta = 1}^\xi  \left[\sup_{\bm S }   \mathbb{P}_0^{\bm S} \left( \frac{\sum\limits_{j= 1}^{ n_b}Z_{\bm{\alpha}^*}[j]  }{n_b}  < I_{\bm{\alpha}^*} -\epsilon \right) \right]^\zeta.\label{eq:ineq_1}
\end{align}
For fixed $\bm{S}$, $b$ define
\begin{align}
I_{\bm{S},b} \triangleq \mathbb{E}_0^{\bm{S}} \left[ \frac{\sum\limits_{j= 1}^{ n_b}Z_{\bm{\alpha}^*}[j]  }{n_b} \right] =  \frac{\sum\limits_{j= 1}^{ n_b} \mathbb{E}_{p_{\bm{S}[j]}} \left[Z_{\bm{\alpha}^*}[j] \right] }{n_b}  \geq I_{\bm{\alpha}}^*,
\end{align}
where the inequality follows from Lemma \ref{equal_drifts_lemma}. This in turn implies that for any $\bm{S}$ we have that
\begin{align}
\label{eq:mia_akoma_equ}
 & \mathbb{P}_0^{\bm S} \left( \frac{\sum\limits_{j= 1}^{ n_b}Z_{\bm{\alpha}^*}[j]  }{n_b}  < I_{\bm{\alpha}^*} -\epsilon \right)  
 \nonumber 
  =
 \mathbb{P}_0^{\bm S} \left( \frac{\sum\limits_{j= 1}^{ n_b}Z_{\bm{\alpha}^*}[j]  }{n_b}  < I_{\bm{\alpha}^*} -\epsilon + I_{\bm{S},b} - I_{\bm{S},b} \right)
 \nonumber \\& \leq \mathbb{P}_0^{\bm S} \left( \frac{\sum\limits_{j= 1}^{ n_b}Z_{\bm{\alpha}^*}[j]  }{n_b}  < I_{\bm{S},b} -\epsilon  \right)
\leq   \mathbb{P}_0^{\bm S} \left(\Bigg|  \frac{\sum\limits_{j= 1}^{ n_b}Z_{\bm{\alpha}^*}[j]}{n_b}  - I_{\bm{S},b}\Bigg| > \epsilon \right).
 \end{align}
 Define 
 \begin{align}
 \bar{\sigma}^2 \triangleq \max_{\bm E \, \in \, \mathcal{E}}  \text{Var}_{p_{\bm E}} \left[\log\frac{\overline{p}_{\bm{\alpha}^*}(\bm{X})}{g(\bm{X})}     \right].
 \end{align}
 From eq. \eqref{eq:assumptio_mom}, we have that $ \bar{\sigma}^2 <\infty$. Then, by Chebychev's inequality 
 \begin{align}
  \mathbb{P}_0^{\bm S} \left(\Bigg|   \frac{\sum\limits_{j= 1}^{ n_b}Z_{\bm{\alpha}^*}[j]}{n_b}  - I_{\bm{S},b}\Bigg| > \epsilon \right)  \nonumber
 &\leq \text{Var}_0^{\bm S} \left( \frac{\sum\limits_{j= 1}^{ n_b}Z_{\bm{\alpha}^*}[j]}{n_b}   \right)\frac{1}{\epsilon^2} \nonumber 
 = \frac{1}{\epsilon^2 n_b^2} \sum_{j=1}^{n_b} \text{Var}_{p_{\bm S[j]}} \left(Z_{\bm{\alpha}^*}[j]   \right)
\\&  \leq \frac{\sum_{j=1}^{n_b}\bar{\sigma}^2}{n_b^2 \epsilon^2} = \frac{\bar{\sigma}^2}{n_b \epsilon^2}. 
\label{eq:ineq_2}
 \end{align}
 By using \eqref{eq:initial_eq}, \eqref{eq:ineq_1}, \eqref{eq:mia_akoma_equ} and \eqref{eq:ineq_2} we then have that
 \begin{align}
&\sup_{\bm S } \mathbb{E}_0^{\bm S}\left[\frac{\tau_W(\bm{\alpha}^*, b)}{n_b}\right] \leq 1+\lim_{\xi \rightarrow \infty}\sum_{\zeta=1}^\xi \left[  \frac{\bar{\sigma}^2}{n_b \epsilon^2}  \right]^\zeta.
\end{align}
Let $0<\delta<1$. Since $n_b$ is increasing with $b$, we have that for all $b>B$, where $B$ large enough
 \begin{align}
&\sup_{\bm S } \mathbb{E}_0^{\bm S}\left[\frac{\tau_C}{n_b}\right] \leq  1+\lim_{\xi \rightarrow \infty}\sum_{\zeta=1}^\xi \delta^\zeta = \sum_{\zeta=0}^\infty \delta^\zeta = \frac{1}{1-\delta}
\end{align}
which implies that for all $b>B$
 \begin{align}
 \label{eq:eqeqeqeq}
&\sup_{\bm S } \mathbb{E}_0^{\bm S}\left[\tau_W(\bm{\alpha}^*, b) \right] \leq   \frac{b}{(I_{\bm{\alpha}^*}-\epsilon)(1-\delta)}.
\end{align}
Since \eqref{eq:eqeqeqeq} holds for all $\epsilon >0$ we have that
 \begin{align}
&\sup_{\bm S } \mathbb{E}_0^{\bm S}\left[\tau_W(\bm{\alpha}^*, b)\right] \leq   \frac{b}{I_{\bm{\alpha}^*}(1-\delta)}.
\end{align}
Finally, since $\delta \rightarrow 0$ as $b\rightarrow \infty$ we have that
 \begin{align}
 \label{eq:ineq_3}
&\mathrm{WADD}(\tau_W(\bm{\alpha}^*, b)) = \sup_{\bm S } \mathbb{E}_0^{\bm S}\left[\tau_W(\bm{\alpha}^*, b)\right] \leq   \frac{b}{I_{\bm{\alpha}}^*}(1+o(1))
\end{align}
as $b \rightarrow \infty$.
\end{proof}
\bibliographystyle{IEEEtran}

\bibliography{CC_bibliography}

\end{document}